\documentclass[11pt]{article}
\usepackage{latexsym,amssymb,amsmath,amsthm,graphicx,float,bm}

\newcommand{\pd}{\partial}
\newcommand{\vf}{\varphi}
\newcommand{\m}{\mathbf{m}}
\newcommand{\n}{\mathbf{n}}
\newcommand{\x}{\mathbf{x}}
\renewcommand{\r}{\mathbf{r}}
\newcommand{\mn}{{\mathbf{m},\mathbf{n}}}

\newcommand{\R}{\mathbb{R}}
\newcommand{\Z}{\mathbb{Z}}

\newcommand{\<}{\langle}
\renewcommand{\>}{\rangle}

\newcommand{\eps}{\epsilon}

\newtheorem{theorem}{Theorem}
\newtheorem{lemma}{Lemma}

\newtheorem{corollary}[theorem]{Corollary}

\title{Scattering in flatland: \\ Efficient representations via wave atoms}

\author{Laurent Demanet$^{\dagger}$ and Lexing Ying$^{\ddagger}$\\
  \vspace{-.1cm}\\
  $\dagger$ Department of Mathematics, Stanford University, Stanford CA94305\\
  $\ddagger$ Department of Mathematics, University of Texas at Austin, Austin, TX 78712 }

\date{May 2008}		% remove for today's date

\begin{document}

\maketitle

\begin{abstract}
This paper presents a numerical compression strategy for the boundary
integral equation of acoustic scattering in two dimensions. These
equations have oscillatory kernels that we represent in a basis of
wave atoms, and compress by thresholding the small coefficients to
zero. 

\begin{figure}[h!]
  \begin{center}
    \includegraphics[height=1.8in]{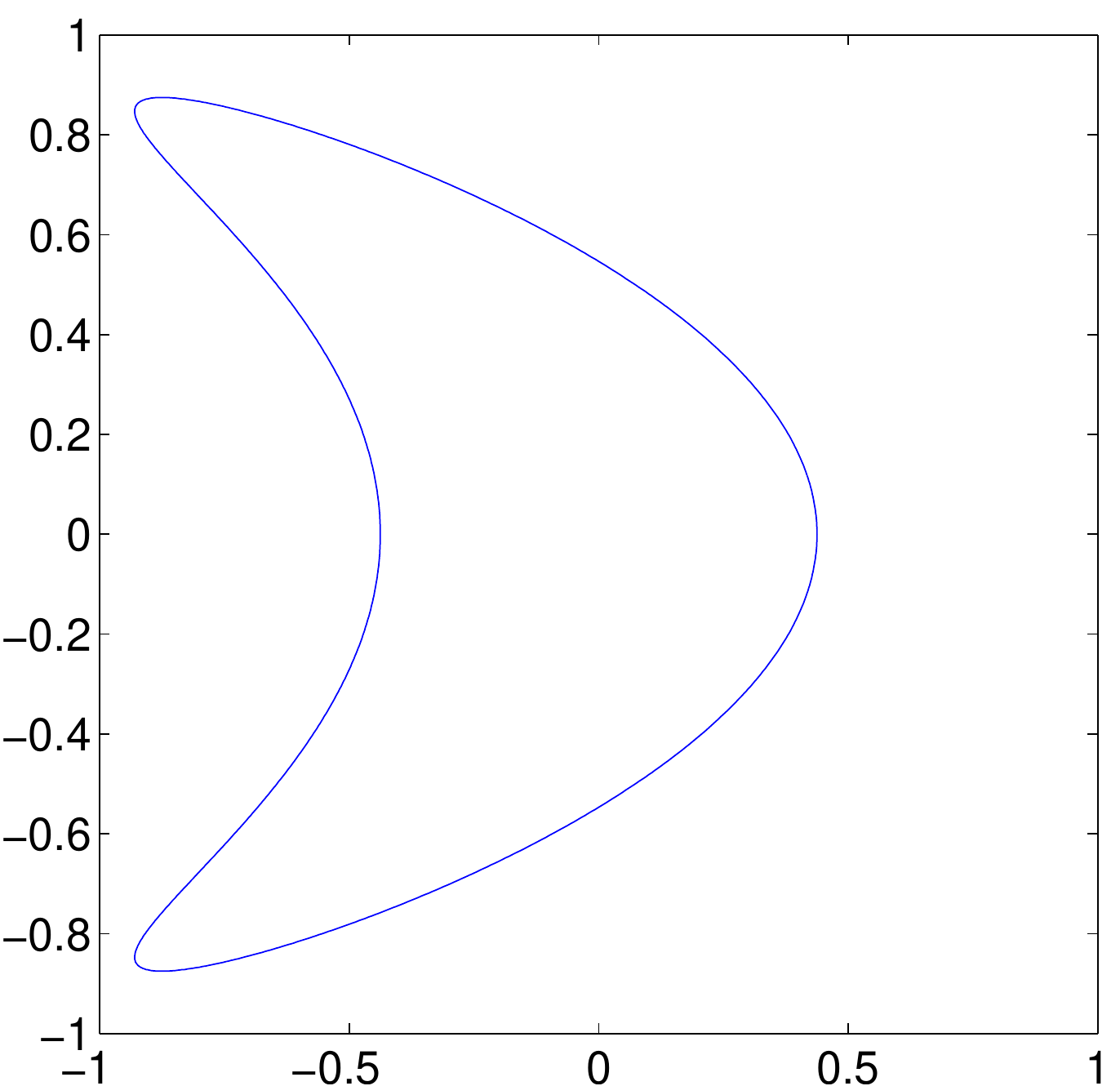}
    \includegraphics[height=1.8in]{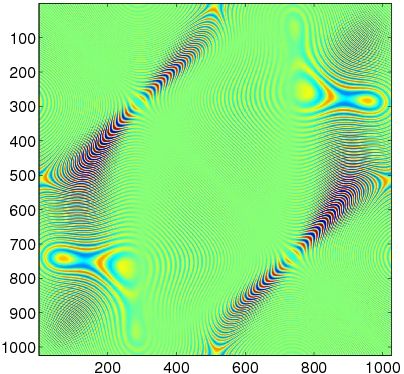}
    \includegraphics[height=1.8in]{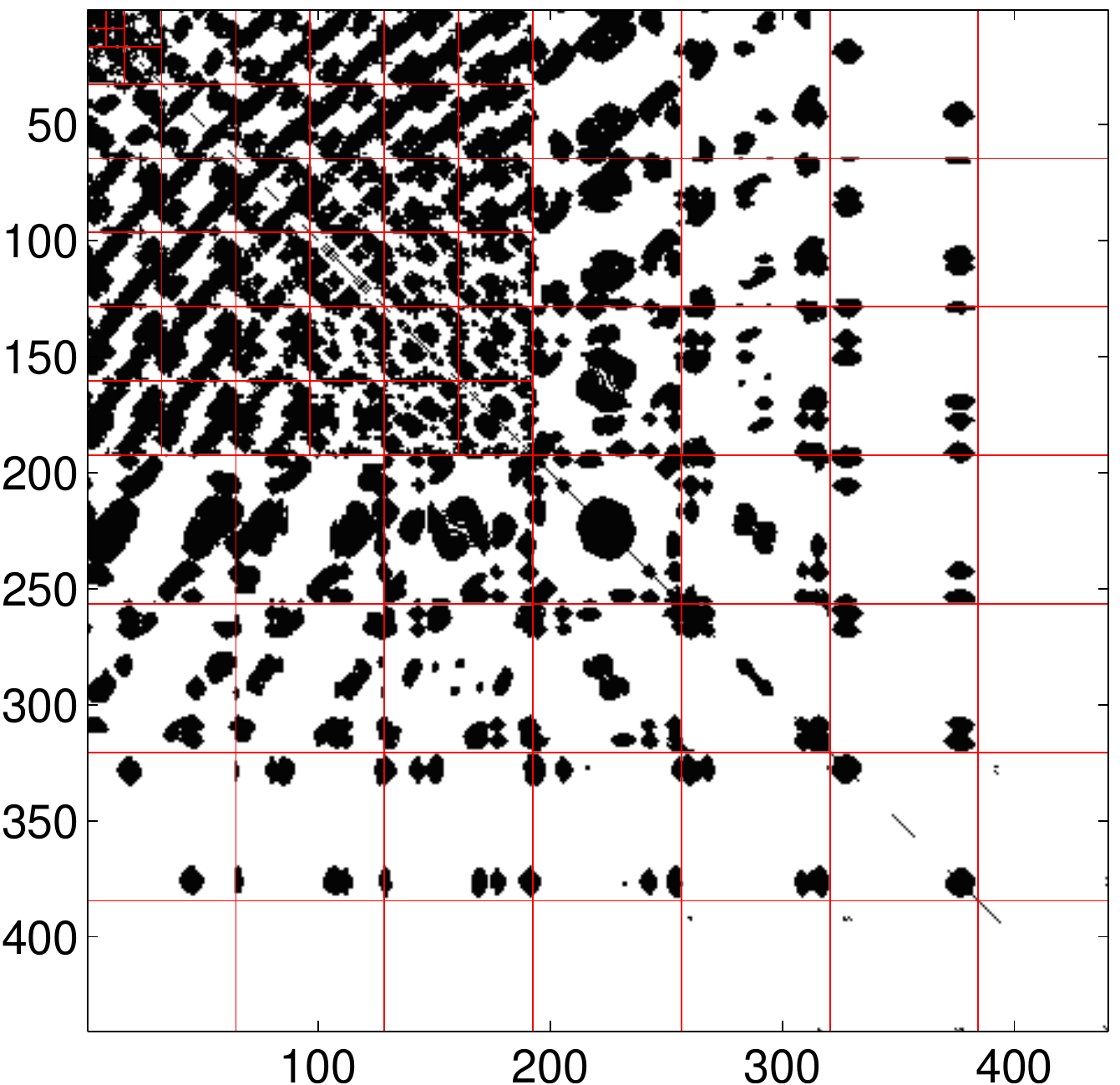}
  \end{center}
\end{figure}

\vspace{-.5cm}

\textbf{Left}: a kite-shaped scatterer. \textbf{Middle}: a depiction of the kernel
of the double-layer potential for this scatterer, sampled as a
1024x1024 matrix. \textbf{Right}: a zoomed-in view of the sparsity pattern of
the ``nonstandard wave atom matrix'', representing this kernel
accurately using only 60,000 matrix elements.

This phenomenon was perhaps first observed in 1993 by Bradie, Coifman,
and Grossman, in the context of local Fourier bases \cite{BCG}.  Their
results have since then been extended in various ways. The purpose of
this paper is to bridge a theoretical gap and prove that a well-chosen
fixed expansion, the nonstandard wave atom form, provides a
compression of the acoustic single and double layer potentials with
wave number $k$ as $O(k)$-by-$O(k)$ matrices with $O(k^{1+1/\infty})$
nonnegligible entries, with a constant that depends on the relative
$\ell_2$ accuracy $\eps$ in an acceptable way. The argument assumes
smooth, separated, and not necessarily convex scatterers in two
dimensions.  The essential features of wave atoms that enable to write
this result as a theorem is a sharp time-frequency localization that
wavelet packets do not obey, and a parabolic scaling wavelength $\sim$
(essential diameter)${}^2$.  Numerical experiments support the
estimate and show that this wave atom representation may be of
interest for applications where the same scattering problem needs to
be solved for many boundary conditions, for example, the computation
of radar cross sections.

\end{abstract}

\bigskip

{\bf Acknowledgements.}
The first author is partially supported by an NSF grant. The second author is
partially supported by an NSF grant, a Sloan Research
Fellowship, and a startup grant from the University of Texas at Austin.

\newpage

%LD: Insert the two pictures?

%-----------------------------------------------------------------------
\section{Introduction}

%LD: scattering field or scattered field?

This paper is concerned with the sparse representation of the boundary
integral operator of two-dimensional scattering problems. Let $D$ be a
bounded soft scatterer in $\R^2$ with a smooth boundary and
$u^{inc}(x)$ be the incoming wave field. The scattered field $u(x)$
satisfies the two-dimensional exterior Dirichlet problem of the
Helmholtz equation:
\[
-\Delta u(x) - k^2 u(x) = 0 \quad\mbox{in}\; \R^d \setminus \bar{D},
\]
\[
u(x) = - u^{inc}(x) \quad\mbox{for}\; x\in\pd D,
\]
\[
\lim_{|x|\rightarrow\infty} |x|^{1/2} \left( \left( \frac{x}{|x|},\nabla u(x) \right) - i k u(x) \right) = 0.
\]
One attractive method for dealing with this problem is to reformulate it using
a boundary integral equation for an unknown field $\phi(x)$ on $\pd
D$:
\begin{equation}
  \frac{1}{2} \phi(x) + \int_{\pd D} \left( \frac{\pd G(x,y)}{\pd n_y} - i \eta G(x,y) \right) \phi(y) d y
  = - u^{inc}(x),
  \label{eq:bie}
\end{equation}
where $n_y$ stands for the exterior normal direction of $\pd D$ at the
point $y$, and $\eta$ is a coupling constant of order $O(k)$. The
kernels $G(x,y)$ and $\frac{\pd G(x,y)}{\pd n_y}$ are, respectively,
the Green's function of the Helmholtz equation and its normal derivative, given by
\[
G(x,y) = \frac{i}{4} H_0^{(1)}(k \| x-y \| ),
\]
and
\[
\frac{\pd G}{\pd n_y}(x,y) = \frac{i k}{4} H_1^{(1)}(k \| x-y \|) \,
\frac{x-y}{\| x- y \|} \cdot n_y.
\]
Once $\phi(x)$ is obtained from solving the integral equation, the
scattered field $u(x)$ at $x\in \R^2\backslash \bar{D}$ can be
evaluated as
\[
u(x) = \int_{\pd D} \left( \frac{\pd G(x,y)}{\pd n_y} - i \eta G(x,y)
\right) \phi(y) d y.
\]
An important property of \eqref{eq:bie} from the computational point
of view is that its condition number is often quite small and, as a
result, one can advantageously solve \eqref{eq:bie} with an iterative algorithm
like GMRES. At each step of the iterative solver, we need to apply
the integral operator to a given function.  Since the integral
operator is dense, applying the operator directly is too expensive. In
this paper, we address this issue by efficiently representing the
operator as a sparse matrix in a system of wave atoms.

Local cosines or wavelet packets have already been proposed for this
task with great practical success, see Section \ref{sec:related} for
some references, but we believe that the following two reasons make a
case for \emph{wave atom frames}:
\begin{itemize}

\item The proposed construction is non-adaptive: wave atom frames of
  $L^2$ are not designed for a specific value of $k$, and no
  optimization algorithm is needed to find a provably good basis. To
  achieve this result, the essential property of wave atoms is a
  parabolic scaling that we discuss later.

\item The choice of numerical realization for wave atoms follows some
  of the experience garnered throughout the 1990s in the study of
  local Fourier bases and wavelet packets. In particular, wave atoms
  offer a clean multiscale structure in the sense that they avoid the
  ``frequency leakage'' associated with wavelet packets defined from
  filterbanks. These aspects are discussed in \cite{DY}.

\end{itemize}

Of course, any non-adaptive all-purpose numerical compression method
is likely to lag in performance behind an adaptive strategy that would
include at least the former in its scope; but this is no excuse for
discarding their study. Proper insight about architectures and
scalings is important for designing the solution around which, for
instance, a library of bases should be deployed for a best basis
search.

The main result of this paper says that the wave atom frame is in some
sense near-optimal for representing the integral operator in
\eqref{eq:bie} as a sparse matrix. Namely, full matrices would involve
$O(k^2)$ elements but we show that $O(k)$ matrix elements suffice to
represent $G$, and $O(k^{1+\delta})$ matrix elements suffice to
represent $\pd G / \pd n_y$ to a given accuracy $\epsilon$, for
arbitrarily small $\delta > 0$. We believe that these bounds would not
hold for wavelet packets, for instance, even if the best decomposition
tree is chosen. In particular, wavelets would obviously not be suited
for the job.

Once the sparse representation of \eqref{eq:bie} is constructed in the
wave atom frame, applying the operator requires only two wave atom
transforms and has complexity $O(N^{1+\delta})$ for a problem with $N$
unknowns ($N$ is proportional to $k$, $\delta$ is arbitrarily small.)
The implicit constant is in practice very small, and the wave atom
transform itself has complexity $O(N \log N)$.

Such a strategy is attractive when the scattering problem needs to be
solved several times with different incoming waves $u^{inc}(x)$. One
important example is the computation of bistatic cross sections, where
one needs to calculate the far field patterns of scattered fields for
all possible incoming plane waves.

% LD: this would be a good place to expand a little on cross-sections;
% could you please be more explicit (1 sentence)?

\subsection{Wave atoms}

Frames of wave atoms were introduced in \cite{DY} on the basis that
they provide sparse representations of certain oscillatory patterns.
As alluded to earlier, they are a special kind of oriented wavelet
packets that do not suffer from the frequency leaking associated to
filterbanks, and which obey the important parabolic scaling relation

\begin{center}
wavelength $\sim$ (essential diameter)${}^2$.
\end{center}

Let us rehearse the construction of wave atoms, and refer the reader
to \cite{DY} for more details. In one dimension, wave atoms are an
orthonormal basis indexed by the triple of integers $\lambda \equiv
(j, m, n)$. The construction is in the frequency domain; our
convention for the Fourier transform is
\[
\hat{f}(\omega) = \int_{\R^n} e^{- i x \cdot \omega} f(x) \, dx, 
\qquad f(x) = \frac{1}{(2 \pi)^n} \int_{\R^n} e^{i x \cdot \omega} \hat{f}(\omega) \, d\omega.
\]

\begin{itemize}
\item First, $j \geq 0$ is a scale parameter that should be thought of
  as indexing dilations \emph{of a factor 4}; in other words, one
  should consider a first partition of the positive frequency axis
  into intervals of the form $[c_1 2^{2j}, c_2 2^{2j + 2}]$ (for some
  constants $c_1, c_2$ that will accommodate overlapping of basis
  functions). Choosing $j$ such that frequency $\omega$ is
  proportional to $2^{2j}$ is in contrast with wavelet theory, where
  $\omega \sim 2^j$ over the positive frequency support of a wavelet.
\item The parameter $m$ with $c_1 2^j \le m < c_2 2^{j+2}$ then indexes
  the further partitioning of each interval $[c_1 2^{2j}, c_2 2^{2j +
    2}]$ into $O(2^j)$ subintervals of size $O(2^{j})$. More
  precisely, wave atoms are centered in frequency near $\pm
  \omega_{\lambda}$ where
  \[
  \omega_{\lambda} \approx \pi 2^j m, \qquad c_1 2^j \le m < c_2 2^{j+2} 
  % \omega_{\lambda} = \pi 2^j m, \qquad m = O(2^j),
  \]
  and are compactly supported in the union of two intervals of length
  $2 \pi \times 2^j$. The parabolic scaling is now apparent; the size
  of the support in frequency $(\sim 2^j)$ is proportional to the
  square root of the offset from the origin $(\sim 2^{2j})$.
\item The parameter $n \in \Z$ indexes translations. A wave atom is
  centered in space near
  \[
  x_\lambda = 2^{-j} n,
  \]
  and has essential support as narrow as the uncertainty principle
  allows, i.e., of length $O(2^{-j})$.
\end{itemize}

We define $\Omega$ to be the set of all admissible indices, i.e.,
\[
\Omega = \{ (j,m,n): j \ge 0, \, c_1 2^j \le m < c_2 2^{j+2}, \, n \in \Z  \}.
\]
Basis functions are then written
\begin{equation}\label{eq:wa1D}
  \vf_{\lambda}(x) = 2^{j/2} \vf_{(j,m)}(2^j x - n) \, e^{i 2^j m x},
  \quad \lambda \in \Omega
\end{equation}
where $\vf_{(j,m)}$ depends weakly on $j$ and $m$, and needs to be
chosen adequately to form an orthobasis. The underlying delicate
construction of the $\vf_{(j,m)}$ is due to Lars Villemoes \cite{Vil} and summarized in
\cite{DY}.

In two dimensions, wave atoms are individually, but not collectively,
formed as tensor products of the one-dimensional basis functions. The
construction is "multiresolution" in the sense that there is only one
dilation parameter; the indexing 5-uple of integers is $\mu \equiv (j,
\m, \n)$ where $\m = (m_1, m_2)$ and $\n = (n_1, n_2)$. More
precisely, at scale $j$, the valid values for $\m = (m_1,m_2)$ satisfy
$0\le m_1, m_2 < c_2 2^{j+2}$ and $ c_1 2^j \le \max(m_1,m_2)$.

Wave atoms come as an orthonormal basis in two dimensions, but can be
made fully directional---supported in a narrow cone in frequency with
apex at the origin \cite{AM}---at the expense of increasing the
redundancy to 2 or 4. The definition of such variants makes use of a
unitary recombination involving Hilbert-transformed basis functions as
in the definition of complex wavelet transforms, and is fully
explained in \cite{DY}.

None of the results of this paper would depend on the choice of
variant; and for convenience we use the frame of wave atoms with
redundancy four.  With $\x = (x_1, x_2)$, the only property of wave
atoms that we will need is the characterization
\begin{equation}\label{eq:vfjm1}
  \vf_\mu(\x) = 2^j \vf_{(j,\m)}(2^j x_1 - n_1, 2^j x_2 - n_2) e^{i 2^j \m \cdot \x},
\end{equation}
where $\vf_{(j,\m)}$ is a $C^\infty$ non-oscillatory bump that depends
on $j$ and $\m$, but in a non-essential manner, i.e.,
\begin{equation}\label{eq:vfjm2}
| \pd^\alpha_{\x} \vf_{(j,\m)}(\x) | \leq C_{\alpha, M}  (1+ \| \x \|)^{-M}, \qquad \forall M > 0,
\end{equation}
with $C_{\alpha, M}$ independent of $j$ and $\m$. In addition, each
$\vf_{(j,\m)}$ is simply the tensor product of two corresponding bumps
for the 1D transform.

Although they may not necessarily form an orthonormal basis, wave
atoms still form a tight frame in the sense that expanding a function
is an isometry from $L^2(\R^2)$ to $\ell_2(\mu)$, 
\[
\| f \|_2^2 = \sum_\mu |\< f, \vf_\mu \>|^2
\]
which is equivalent to
\begin{equation}\label{eq:tightframe}
f = \sum_\mu \< f, \vf_\mu \> \vf_\mu.
\end{equation}
The same properties hold in one dimension.

The closest analogue to a ``continuous wave atom transform'' was
introduced in the mathematical literature by C\'{o}rdoba and Fefferman
in \cite{CF}. Wave atoms can be compared to brushlets
\cite{Brushlets}, but Villemoes's construction uses ``local complex
exponentials'' instead of local cosines in frequency.

Discretized wave atoms are described in \cite{DY, mirror-extended};
they inherit the localization and tight-frame properties of their
continuous counterpart. In particular, their bandlimited character
confers an immediate control over the accuracy of computing inner
products via quadrature. They come with fast FFT-based $O(N \log N)$
algorithms for both the forward and adjoint transforms (see \cite{DY}
for details).

%(Here and in what follows, the notation $A \leq B^{-\infty}$ means $A \leq C_M B^{-M}$ for all $M > 0$.)

%the successive derivatives of $\vf_{(j,\m)}(\x)$ can be uniformly bounded by constants independent of $j$ and $\m$.

\subsection{Operator expansions}

As functions can be analyzed and synthesized using coefficients,
operators can also be expanded from matrix elements in a tight frame.
When the frame is chosen appropriately, the expansion of the
operator is approximately sparse: only a small percentage of the
expansion coefficients are nonnegligible while all the rest of the
coefficients are almost zero. Thresholding the small coefficients
below a certain accuracy level gives rise to a sparse representation
of the operator. Typically, this can be done in two ways.

\begin{itemize}
\item The \emph{standard form} of an operator $A$ in the wave atom frame
$\vf_{\lambda}$ is
\[
A = \sum_{\lambda\in\Omega} \sum_{\lambda'\in\Omega} \vf_\lambda A_{\lambda,\lambda'} \< \cdot, \vf_{\lambda'} \>,
\]
where $ A_{\lambda,\lambda'} = \< A \vf_{\lambda'}, \vf_\lambda \>.$
For each fixed $\lambda$, we define $S(\lambda)$ to be the set of all
$\lambda'$ such that the modulus of $A_{\lambda,\lambda'}$ is above a certain 
threshold. The sparse representation of $A$ then takes the form
\[
A \approx \sum_{\lambda} \vf_\lambda \sum_{\lambda' \in S(\lambda)}
A_{\lambda,\lambda'} \<\cdot,\vf_{\lambda'} \>.
\]
In practice, the sums in $\lambda$ and $\lambda'$ are also truncated in scale to account for the finite number of samples $N$ of the functions to which $A$ is applied. Since the wave atom transform
computes the coefficients in the tight frame $\vf_\lambda$ in $O(N\log
N)$ steps, the above equation naturally gives rise to an efficient
method of applying the operator $A$ to a given function $f$:
\begin{itemize}
\item Apply the forward wave atom transform to compute the
  coefficients $f_{\lambda'} := \<f,\vf_{\lambda'} \>$.
\item For each $\lambda$, compute $g_\lambda := \sum_{\lambda' \in
    S(\lambda)} A_{\lambda,\lambda'} f_{\lambda'}$. The number of
  operations of this step is dictated by the number of nonnegligible
  coefficients of the representation of $A$.
\item Apply the adjoint wave atom transform to $g_\lambda$ to
  synthesize $Af$, i.e., $Af \approx \sum_\lambda \vf_\lambda
  g_\lambda$.
\end{itemize}

\item The \emph{nonstandard form} of $A$ in the two dimensional frame
$\vf_\mu$ is the set of coefficients $A_\mu = \int_{\R^2} A(x_1,x_2)
\overline{\vf}_\mu(x_1,x_2) \, dx_1dx_2$, such that the distributional
kernel of $A$ is expanded as
\[
A(x_1,x_2) = \sum_{\mu} A_\mu \vf_\mu(x_1,x_2)
\]
For a fixed threshold value, we define $S$ to be the set of all $\mu$
such that $A_\mu$ is above the threshold in modulus.  The sparse
representation of $A$ is now $ A(x_1,x_2) \approx \sum_{\mu\in S}
A_\mu \vf_\mu(x_1,x_2) $. Applying $A$ to a given function $f$
efficiently using this expansion is more involved than the case of the
standard form. For a fixed index $\mu=(j,\m,\n)$ with $\m=(m_1,m_2)$
and $\n=(n_1,n_2)$, we define the two 1D wave atom indices:
\begin{equation}
  \lambda^\mu_1 = (j,m_1,n_1) \quad\mbox{and}\quad
  \lambda^\mu_2 = (j,m_2,n_2).
  \label{eq:newind}
\end{equation}
Since the two dimensional index $\m=(m_1,m_2)$ satisfies $0 \le m_1, m_2 <
c_2 2^{j+2}$ and $ c_1 2^j \le \max(m_1,m_2)$, the set of all possible
choices for $\lambda^\mu_1$ and $\lambda^\mu_2$ are
\begin{equation}
  0 \le m_1 < c_2 2^{j+2}, \quad\mbox{and}\quad
  0 \le m_2 < c_2 2^{j+2}.
\label{eq:allind}
\end{equation}

Some of the indices in \eqref{eq:allind} are not admissible, i.e.,
they are not part of the set of indices for the 1D wave atom
transform, since if $\mu = (j,m,n)$ corresponds to a 1D wave atom it
would need to satisfy $c_1 2^j \le m < c_2 2^{j+2}$. Non-admissible
indices correspond to Gabor-type wave forms that partition the
frequency domain uniformly and, hence, violate the parabolic scaling.
We use $\Omega^e$ to denote this extended index set
\[
\Omega^e = \{ (j,m,n): j\ge 0, 0 \le m < c_2 2^{j+2}, n \in \Z \}.
\]
The frame formed by these bases functions $\vf_\lambda$ with $\lambda
\in \Omega^e$ are called the {\em extended} wave atom frame. For a
given function $f$, the computation of all the coefficients $\<f,
\vf_\lambda\>$ with $\lambda \in \Omega^e$ can be done easily by
extending the existing forward wave atom transform to include the
extra $m$ indices ($0\le m < c_1 2^j$) in $\Omega^e$. This
\emph{forward extended wave atom transform} still has an $O(N\log N)$
complexity since the number of extra non-admissible indices $m$ at each scale $j$ is
a small fixed fraction of the number of existing $m$ in $\Omega$ at scale $j$. The adjoint transform
can be extended similarly and the resulting transform is called the
{\em adjoint extended wave atom transform}. We note that these new
transforms are not orthonormal any more since parts of the input
functions are analyzed redundantly.

Using the notation in \eqref{eq:newind} and the tensor-product
property of $\vf_\mu$, we have
\begin{equation}
  A(x_1,x_2) \approx \sum_{\mu\in S} A_\mu \vf_{\lambda^\mu_1}(x_1) \vf_{\lambda^\mu_2}(x_2).
  \label{eq:tensor}
\end{equation}
When $A$ is applied to a given function $f$, we have
\[
A f(x_1) 
\approx \sum_{\mu\in S} A_\mu \vf_{\lambda^\mu_1}(x_1) \left( \int \vf_{\lambda^\mu_2}(x_2) f(x_2) d x_2 \right) 
= \sum_{\lambda \in \Omega^e} \vf_{\lambda} 
\sum_{\mu\in S\;\mbox{s.t.}\;\lambda^\mu_1=\lambda} A_\mu \left( \int \vf_{\lambda^\mu_2}(x_2) f(x_2) d x_2 \right).
\]
Using the extended transforms, we can derive from the above equation a
fast algorithm for applying $A$ to $f$ using the nonstandard form:
\begin{itemize}
\item Apply the forward extended wave atom transform to compute the
  coefficients $f_{\lambda'} := \<f,\vf_{\lambda'} \>$ for all indices
  $\lambda'\in \Omega^e$.
\item For each $\lambda \in \Omega^e$, compute $g_\lambda :=
  \sum_{\mu\in S\;\mbox{s.t.}\;\lambda^\mu_1=\lambda} A_\mu
  f_{\lambda^\mu_2}$.  The number of operations in this step is
  proportional to the number of nonnegligible coefficients in the
  nonstandard expansion of the operator $A$.
\item Apply the inverse extended wave atom transform to synthesize
  $Af$ from $g_\lambda$, i.e., $Af \approx \sum_{\lambda\in\Omega^e}
  \vf_\lambda g_\lambda$.
\end{itemize}

We would like to point out that nonstandard expansions only exist for
two-dimensional frames that have a tensor product representation for
each basis function, since the decomposition in \eqref{eq:tensor} is
essential for the derivation. For example, there is no known
nonstandard form representation for an operator in the tight frame of
curvelets \cite{CD}.

\end{itemize}

In what follows we focus exclusively on the nonstandard form, because
of its relative simplicity over the standard form. This claim may seem
paradoxical in view of the preceding discussion, but the basis
functions $\vf_\mu(x_1,x_2)$ of the nonstandard form have a single
dilation parameter and hence a fixed isotropic aspect ratio, unlike
the tensor basis $\vf_\lambda(x_1) \vf_{\lambda'}(x_2)$ which come in
all shapes and can be very elongated when $j \ne j'$. The isotropy of
the envelope of $\vf_\mu$ makes some of the stationary-phase argument
in the sequel, simpler in our view. The kernel $A(x_1,x_2)$ does not
display the anisotropy of tensor wave atoms along the axes---if
anything, $A(x_1,x_2)$ is singular along the diagonal $x_1 = x_2$---so
choosing the nonstandard form also seems to us more natural from the
numerical viewpoint. We can however not exclude at this point that the
standard form may enjoy comparable sparsity properties as the
nonstandard form.

\subsection{Sparsity of the nonstandard wave atom matrix}

In this section, we formulate the main result on sparsity of the nonstandard wave atom matrix of the acoustic single and double-layer potentials, in two dimensions.

The scatterer is a union of closed, nonintersecting $C^\infty$ curves $\Omega = \bigcup_{\alpha=1}^n \Omega_\alpha$ embedded in $\R^2$. For each $\alpha$, assume that $\mathbf{x}(t): I_\alpha \mapsto \Omega_\alpha$ is a $C^\infty$ periodic parametrization of $\Omega_\alpha$, and take $I_\alpha = [0,1]$ for simplicity.

We assume the following mild \emph{geometric regularity} condition on the scatterer: there exists $D > 0$ such that
\begin{equation}\label{eq:geomreg}
\| \x(s) - \x(t) \| \geq D \, | e^{2\pi i s} - e^{2 \pi i t} |,
\end{equation}
essentially meaning that the curve $\Omega_\alpha$ defining the scatterer cannot  intersect itself. 
We write $d(s,t) \equiv | e^{2\pi i s} - e^{2 \pi i t} |$ for the Euclidean distance on the unit circle.

When $s \in I_\alpha$, $t \in I_\beta$ with $1 \leq \alpha, \beta \leq n$, let
\begin{equation}\label{eq:SLP}
G_0(s,t) = \frac{i}{4} H^{(1)}_0(k \| \mathbf{x}(s) - \mathbf{x}(t) \|) \, \|\dot{\mathbf{x}}(t)\|,
\end{equation}
and
\begin{equation}\label{eq:DLP}
G_1(s,t) = \frac{i k}{4} H_1^{(1)}(k \| \mathbf{x}(s)-\mathbf{x}(t) \|) \, \frac{\mathbf{x}(s) -\mathbf{x}(t)}{\| \mathbf{x}(s) - \mathbf{x}(t) \|} \cdot n_{\mathbf{x}(t)} \, \|\dot{\mathbf{x}}(t)\|.
\end{equation}

The nonstandard wave atom matrices of $G_0$ and $G_1$, restricted to a couple of intervals $I_\alpha \times I_\beta$, are
\[
K^0_\mu = \< G_0, \vf_\mu \>, \qquad K^1_\mu = \< G_1, \vf_\mu \>. 
\]
Our main result below concerns the existence of \emph{$\eps$-approximants} $\tilde{K}^0_\mu$ and $\tilde{K}^1_\mu$, corresponding to the restriction of $\mu$ to sets $\Lambda_0$ and $\Lambda_1$, i.e., with $a = 0,1$,

\[
\tilde{K}^a_{\mu} = \left\{ \begin{array}{ll}
    K^{a}_{\mu} & \mbox{if $\mu \in \Lambda_{a}$};\\
    0 & \mbox{otherwise,} \end{array} \right.
\]
and chosen by definition such that
\begin{equation}
  \| K^a - \tilde{K}^a \|_{\ell_2(\mu)} \leq \eps, \qquad a = 0,1.
  \label{eq:Karel}
\end{equation}
%\[
%\| K^a - \tilde{K}^a \|_{\ell_2} \leq \eps, \qquad a = 0,1.
%\]
The $\ell_2$ norm of a nonstandard wave atom matrix is equivalent to a Hilbert-Schmidt norm for the corresponding operator, by the tight-frame property of wave atoms. This norm is of course stronger than the operator $L^2$-to-$L^2$ norm; and much stronger than the $\ell_\infty(\mu)$ norm used in \cite{BCG}.

In what follows the notation $A \lesssim B$ means $A \leq C \, B$ for some constant $C$ that depends only on the nonessential parameters. Similarly, the notation $A \lesssim \eps^{-1/\infty}$ means $A \leq C_M \eps^{-1/M}$ for all $M > 0$. The constants may change from line to line.

\bigskip

\begin{theorem}\label{teo:main}
Assume the scatterer is smooth and geometrically regular in the sense of (\ref{eq:geomreg}). In the notations just introduced, there exist sets $\Lambda_0$ and $\Lambda_1$ that define $\eps$-approximants of $K^0_\mu$ and $K^1_\mu$ respectively, and whose cardinality obeys
\begin{equation}\label{eq:cardinality0}
| \Lambda_0 | \leq C^0_M \,  \left[ \, k  \, \eps^{-1/M} + \left( \frac{1}{\eps} \right)^{2+1/M} \, \right],
\end{equation}
\begin{equation}\label{eq:cardinality1}
| \Lambda_1 | \leq C^1_M \,  \left[ \, k^{1+1/M}  \, \eps^{-1/M} + \left( \frac{k}{\eps} \right)^{2/3 + 1/M} \, \right],
\end{equation}
for all $M > 0$, and where $C^0_M$, $C^1_M$ depend only on $M$ and the geometry of the scatterers.

\end{theorem}

\bigskip

%The norm $\| K - \tilde{K} \|_{HS}$ is the Hilbert-Schmidt norm of $K$ as an operator, i.e., the $L^2$ norm of the kernel of $K - \tilde{K}$ as a function of $s$ and $t$, or equivalently in our context, the $\ell_2$ norm of $K_{\mu} - \tilde{K}_\mu$ as coefficients indexed by $\mu$, in the wave atom tight frame.

The terms proportional to $k$ or $k^{1+1/\infty}$ are due to the
oscillations when $s \ne t$, and the terms $\eps^{-2-1/\infty}$ and
$(k/\eps)^{2/3+1/\infty}$ are due to the kernels' singularities on the
diagonal $s = t$. While the growth rate of $k$ for the oscillations
term is smaller for $G_0$ than for $G_1$, the growth rate of the
diagonal contribution is smaller for $G_1$ than for $G_0$ since
\[
k^{2/3} \eps^{-2/3} = k^{1/3} k^{1/3} \eps^{-2/3} \leq \max(k, k, \eps^{-2}) \leq 3 (k + \eps^{-2}).
\]

Theorem \ref{teo:main} can also be formulated using \emph{relative} errors instead of absolute errors. This viewpoint is important for considering the composed kernel
\[
G_{(0,1)}(s,t) = G_1(s,t) - i \eta G_0(s,t), \qquad \eta \asymp k.
\]
Call $K^{(0,1)}_\mu$ the nonstandard wave atom matrix of $G_{(0,1)}$. In the following result, we quantify the number of terms needed to obtain the relative error estimate
\[
  \| K^a - \tilde{K}^a \|_{\ell_2(\mu)} \leq \eps \| K^a \|_{\ell_2(\mu)}, \qquad a = 0,1, \mbox{ or } (0,1).
\]

\begin{corollary}\label{teo:cor}
Let $\eta \asymp k$. In the assumptions and notations of Theorem \ref{teo:main}, let $\Delta_0$, $\Delta_1$ and $\Delta_{(0,1)}$ be the sets of wave atom coefficients needed to represent the operators $G_0$, $G_1$, resp. $G_{(0,1)}$ up to relative accuracy $\eps$. Then
\begin{equation}\label{eq:G0rel}
| \Delta_0 | \leq C^0_M \, \left[ ( k \eps^{-2})^{1+1/M} \right], 
\end{equation}
\begin{equation}\label{eq:G1rel}
| \Delta_1 | \leq C^1_M \, \left[ k^{1+1/M} \eps^{-1/M} + ( k \eps^{-2})^{1/3+1/M} \right], 
\end{equation}
\begin{equation}\label{eq:G0G1rel}
| \Delta_{(0,1)} | \leq C^{(0,1)}_M \, \left[ (k \eps^{-2})^{1+1/M} \right],
\end{equation}
for all $M > 0$, and where the constants depend only on $M$ and on the geometry of the scatterers.

\end{corollary}

We do not know if factors such as $k^{1/\infty}$ and $\eps^{-1/\infty}$ could be replaced by log factors.

The proofs of Theorem \ref{teo:main} and Corollary \ref{teo:cor} occupy Section
\ref{sec:sparsity}. The main ingredients are sparsity estimates in
$\ell_p$, stationary phase considerations, vaguelette-type estimates
adapted to wave atoms, and $\ell_2$ correspondence scale-by-scale with
wavelets. In Section \ref{sec:num} we present some numerical
experiments that support the theory and establish wave atoms as a
practical tool for solving scattering problems.

\subsection{Related work}\label{sec:related}

There has been a lot of work on sparsifying the integral operator of
\eqref{eq:bie}, or some variants of it, in appropriate bases. In \cite{BCG}, Bradie et al.
showed that the operator becomes sparse in a local cosine basis. They
proved that the number of coefficients with absolute value greater
than any fixed $\eps$ is bounded by $O(k \log k)$ when the constant
depends on $\eps$. Notice that our result in Theorem \ref{teo:main} is
stronger as the $\ell_2$ norm is used instead in \eqref{eq:Karel}. In
\cite{ABCIS}, Averbuch et al. extended the work in \cite{BCG} by
performing best basis search in the class of adaptive hierarchical
local cosine bases.

Besides the local cosine transform, adaptive wavelet packets have been
used to sparsify the integral operator as well. Deng and Ling
\cite{DL1} applied the best basis algorithm to the integral operator
to choose the right one dimensional wavelet packet basis. Golik
\cite{G} independently proposed to apply the best basis algorithm on
the right hand side of the integral equation \eqref{eq:bie}. Shortly
afterwards, Deng and Ling \cite{DL2} gave similar results by using a
predefined wavelet packet basis that refines the frequency domain near
$k$. All of these approaches work with the standard form expansion of
the integral operator. Recently in \cite{H}, Huybrechs and Vandewalle
used the best basis algorithm for two dimensional wavelet packets to
construct a nonstandard sparse expansion of the integral operator. In
all of these results, the numbers of nonnegligible coefficients in the
expansions were reported to scale like $O(k^{4/3})$.  However, our
result shows that, by using the nonstandard form based on wave atoms,
the number of significant coefficients scales like
$O(k^{1+1/\infty})$.

Most of the approaches on sparsifying \eqref{eq:bie} in well-chosen
bases require the construction of the full integral operator. Since
this step itself takes $O(k^2)$ operations, it poses computational
difficulty for large $k$ values. In \cite{BCR}, Beylkin et al.
proposed a solution to the related problem of sparsifying the boundary
integral operator of the Laplace equation. They successfully avoided
the construction of the full integral operator by predicting the
location of the large coefficients and applying a special one-point
quadrature rule to compute the coefficients. The corresponding
solution for the integral operator of the Helmholtz equation is still
missing.

There has been a different class of methods, initiated by Rokhlin in
\cite{R1,R2}, that requires no construction of the integral operator
and takes $O(k\log k)$ operations in 2D to apply the integral
operator. A common feature of these methods \cite{Ch,EY1,EY2,R1,R2} is
that they partition the spatial domain hierarchically with a tree
structure and compute the interaction between the tree nodes in a
multiscale fashion: Whenever two nodes of the tree are well-separated,
the interaction (of the integral operator) between them is either
accelerated either by Fourier transform-type techniques
\cite{Ch,R1,R2} or by directional low rank representations
\cite{EY1,EY2}. 

A criticism of the methods in \cite{Ch,EY1,EY2,R1,R2} is that the
constant in front of the complexity $O(k \log k)$ is often quite high.
On the other hand, since the FFT-based wave atom transforms are
extremely efficient, applying the operator in the wave atom frame has
a very small constant once the nonstandard sparse representation is
constructed.  Therefore, for applications where one needs to solve the
same Helmholtz equation with many different right hand sides, the
current approach based on the wave atom basis can offer an competitive
alternative. As mentioned earlier, one important example is the computation of
the radar cross section.

%----------------------------------------------------------------------------------
\section{Sparsity analysis}\label{sec:sparsity}

This section contains the proof of Theorem \ref{teo:main}. The
overarching strategy is to reduce the $\ell_2$ approximation problem
to an estimate of $\ell_p$ sparsity through a basic result of
approximation theory, the direct ``Jackson'' estimate
\[
\| K_\mu - \tilde{K}_\mu \|_{2} \leq C \, |\Lambda|^{\frac{1}{2}-\frac{1}{p}} \, \| K_\mu \|_{p},
\]
where $\| K_\mu \|_p^p = \sum_\mu |K_\mu|^p$. Here $\tilde{K}_\mu$ refers to the approximation of $K_\mu$ where only the $|\Lambda|$ largest terms in magnitude are kept, and the others put to zero. The inequality is valid for all values of $0 < p < 2$ for which $\| K_\mu \|_{p}$ is finite. For a proof, see \cite{Mal}, p. 390.

If the $\ell_2$ error is to be made less than $\eps$, it is enough to have $K_\mu$ in some $\ell_p$ space, $0 < p < 2$, and take the number of terms defining $\tilde{K}_\mu$ to be
\begin{equation}\label{eq:Lambda-ellp}
|\Lambda| \geq C_p \, \eps^{\frac{2p}{p-2}} \, \| K_\mu \|^{\frac{2p}{2-p}}_{p}
\end{equation}
for some adequate $C_p > 0$. The sequence $K_\mu$ will be split into several fragments that will be studied independently. For each of these fragments $F$ in $\mu$ space, the inequality (\ref{eq:Lambda-ellp}) will be complemented by an estimate of the form $\| K_\mu \|_{\ell_p(F)} \leq C_p k^{q(p)}$, for all $p > p_0$. Three scenarios will occur in the sequel:

\begin{itemize}
\item If
\begin{equation}\label{eq:Kmu-ellp1}
p_0 = 0 \quad \mbox{and} \quad q(p) = \frac{1}{p} - \frac{1}{2},
\end{equation}
then $|F| \lesssim k \eps^{-1/\infty}$, which is the first term in (\ref{eq:cardinality0}).
\item If
\begin{equation}\label{eq:Kmu-ellp2}
p_0 = 1 \quad \mbox{and} \quad q(p) = 0,
\end{equation}
then $|F| \lesssim \eps^{-2-1/\infty}$, which is the second term in (\ref{eq:cardinality0}).
\item If
\begin{equation}\label{eq:Kmu-ellp3}
p_0 = \frac{1}{2} \quad \mbox{and} \quad q(p) = \frac{1}{p} - 1 + \delta,
\end{equation}
for arbitrarily small $\delta > 0$, then $|F| \lesssim (k/\eps)^{2/3+1/\infty}$, which is the second term in (\ref{eq:cardinality1}).
\end{itemize}

The problem is therefore reduced to identifying contributions  in the sequence $K_\mu$ that obey either of the three estimates above. In what follows we focus on the kernel $K = G_0$. We mention in Section \ref{sec:dlp} how the proof needs to be modified to treat the kernel $G_1$.

\subsection{Smoothness of Hankel functions}\label{sec:hankel}

Bessel and Hankel functions have well-known asymptotic expansions near the origin and near infinity. That these asymptotic behaviors also determine smoothness in a sharp way over the whole half-line is perhaps less well-known, so we formulate these results as lemmas that we prove in the Appendix.

\begin{lemma}\label{teo:hankel1}
For every integers $m \geq 0$ and $n \geq 0$ there exists $C_{m,n}  > 0$ such that for all $k > 0$, 
\begin{equation}\label{eq:nonosc}
| \left( \frac{d}{dx} \right)^m \left[ e^{-ikx} H_n^{(1)}(kx) \right] | \leq \left\{ \begin{array}{ll}
C_{mn} \, (kx)^{-1/2} \, x^{-m} & \mbox{if $kx \geq 1$}; \\
C_{mn} \, (kx)^{-n} \, x^{-m} & \mbox{if $0< kx < 1$ and $m + n > 0$}; \\
C \, (1 + | \log kx |) & \mbox{if $0 < kx < 1$ and $m = n = 0$}.\end{array} \right.
\end{equation}
The same results hold if $1$ is replaced by any number $c > 0$ in $kx < 1$ vs. $kx \geq 1$.
\end{lemma}

The point of equation (\ref{eq:nonosc}) is that $C_{m,n}$ is independent of $k$. Slightly more regularity can be obtained near the origin when multiplying with the adequate power of $x$, as the following lemma shows in the case of $H^{(1)}_1$.

\begin{lemma}\label{teo:hankel2}
For every integer $m \geq 0$ there exists $C_{m}  > 0$ such that, for $0 < x \leq 1$,
\begin{equation}\label{eq:nonosc2}
| \left( \frac{d}{dx} \right)^m \left[ x H_1^{(1)}(x) \right] | \leq \left\{ \begin{array}{ll}
C_{m} & \mbox{if $m = 0,1$}; \\
C_{2} \, (1 + |\log x|) & \mbox{if $m = 2$}; \\
C_m \, x^{2-m} & \mbox{if $m > 2$}.\end{array} \right.
\end{equation}
\end{lemma}

Finally, we will need the following lower bound.

\begin{lemma}\label{teo:hankel3}
For each $n \geq 0$, there exist $c_n > 0$ and $C_n > 0$ such that, when $x > c_n$,
\[
|H_n^{(1)}(x)| \geq C_n x^{-1/2}.
\]
\end{lemma}

%\begin{lemma}\label{teo:hankel2}
%a
%\end{lemma}

%Following this Lemma, we focus without loss of generality on the Hankel function $H_0^{(1)}$ in the sequel.

%\begin{lemma}\label{eq:}
%There exist $C_0 > 0$ and $C_1 > 0$ such that, for kx < 1, 
%\begin{equation}\label{eq:hankel-nearorigin}
%| H_0^{(1)}(kx) | \leq C_0 \cdot | \log x |, \qquad | x H_1^{(1)}(x)| \leq C \qquad \mbox{when } 0 \leq x < 1.
%\end{equation}

\subsection{Dyadic partitioning}

Consider $K$ as in Theorem \ref{teo:main}, with $s \in I_\alpha$ and $t \in I_\beta$. If $\alpha = \beta$, $K$ presents a singularity on its diagonal, whereas if $\alpha \ne \beta$ it presents no such singularity. The case $\alpha = \beta$ is representative and is treated in the sequel without loss of generality.

In this section we assume, as we have above, that $I_\alpha = [0,1]$. The first step of the proof is to partition the periodized square $I_\alpha \times I_\alpha$ at each scale $j$, into dyadic squares denoted
\[
Q = [2^{-j} q_1, 2^{-j} (q_1+1)] \times [2^{-j} q_2, 2^{-j} (q_2+1)].
\]
We define $w_Q$ a window localized near $Q$ through
\[
w_Q(s,t) = w(2^j s - q_1, 2^j t - q_2),
\]
where $w$ is compactly supported on $[-1,2]^2$ and of class $C^\infty$. As a result $w_Q$ is compactly supported inside
\[
3Q \equiv [2^{-j} (q_1-1), 2^{-j} (q_1+2)] \times [2^{-j} (q_2-1), 2^{-j} (q_2+2)].
\]
We also write $x_Q = (2^{-j} q_1, 2^{-j} q_2)$ for the bottom-left corner of $Q$, not to be confused with $\mathbf{x}$, which is in physical space.

Denote by $\mathcal{Q}_j$ the set of dyadic squares at scale $j$; we assume that $w$ is chosen so that we have the scale-by-scale partition of unity property
\[
\sum_{Q \in \mathcal{Q}_j} w_Q = 1.
\]

The kernel is now analyzed at each scale $j$ as $K = \sum_{Q \in \mathcal{Q}_j} K_Q$, where $K_Q = w_Q K$. Dyadic squares can be classified according to their location with respect to the diagonal $s=t$, where $K$ is singular.

\begin{enumerate}
\item \emph{Diagonal squares.} Dyadic squares will be considered ``diagonal squares'' as soon as the distance from their center to the diagonal $s=t$ is less than $1/k$. Scale-by-scale, this condition reads $d(2^{-j }q_1, 2^{-j} q_2) \leq 3 \max ( 2^{-j}, \frac{1}{k} )$. (We need to use of the circle distance $d$ since $q_1$ and $q_2$ are defined modulo $2^j$.) There are $ O( 2^j  \max ( 1, 2^{j} / k ) )$ such diagonal squares at scale $j$. They correspond to the case $kx \lesssim 1$ in Lemma \ref{teo:hankel1}.

\item \emph{Nondiagonal squares.} When $d(2^{-j} q_1, 2^{-j} q_2) > 3 \max ( 2^{-j}, 1 / k )$, we say the square is nondiagonal. in those squares, the kernel $K_Q$ is $C^\infty$ but oscillatory. There are $O(2^{2j})$ such nondiagonal squares at scale $j$. They correspond to the case $kx \gtrsim 1$ in Lemma \ref{teo:hankel1}.

\end{enumerate}

We take the scale $j$ of the dyadic partitioning to match the scale $j$ in the wave atom expansion; namely if $\mu = (j,\m,\n)$, then
\[
K_{\mu} = \< K, \vf_\mu \> = \sum_{Q \in \mathcal{Q}_j} \< K_Q, \vf_\mu \>.
\]
When $0 < p \leq 1$, an estimate on the total $\ell_p$ norm can then be obtained from the $p$-triangle inequality, as follows:
\begin{equation}\label{eq:ptriangle-KQ}
\sum_j \sum_{\m} \sum_{\n} |K_{j,\m,\n}|^p \leq \sum_j \sum_{Q \in \mathcal{Q}_j} \sum_{\m} \sum_{\n} | \< K_Q, \vf_{j,\m,\n} \> |^p.
\end{equation}
When $p \geq 1$, then the regular triangle inequality will be invoked instead, for instance as in
\begin{equation}\label{eq:triangle-KQ}
\left( \sum_j \sum_{\m} \sum_{\n} |K_{j,\m,\n}|^p \right)^{1/p} \leq \sum_j \sum_{Q \in \mathcal{Q}_j} \left( \sum_{\m} \sum_{\n} | \< K_Q, \vf_{j,\m,\n} \> |^p \right)^{1/p}.
\end{equation}

The rationale for introducing a partitioning into dyadic squares is the technical fact that wave atoms are not built compactly supported in space. The windows $w_Q$ allow to cleanly separate different regions of the parameter patch in which the kernel $K$ oscillates with different local wave vectors. 

Note also that the dyadic partitioning is a mathematical tool for the proof of Theorem \ref{teo:main}, and is not part of the construction of the wave atom transform.

\subsection{Geometry of stationary phase points}

Lemma \ref{teo:hankel1} identifies the argument of the Hankel function as a \emph{phase}. For the kernel $K$, this phase is  $k \phi(s,t)$ where $\phi(s,t) = \| \x(s) - \x(t) \|$, and generates typical oscillations as long as $\nabla \phi$ has large magnitude. In this section, we argue that the locus of near-critical (or near-stationary) points of $\phi$ necessarily has small measure. The following lemma makes this heuristic precise in terms of the \emph{scale defect} $j'$.

\begin{lemma}\label{teo:scaledefect}
Let $\phi(s,t) = \| \mathbf{x}(s) - \mathbf{x}(t) \|$ for $s, t$ in some $I_\alpha$. For $j' \geq 0$, let
\[
\mathcal{K}_{j}(j') = \{ (q_1,q_2): \| \nabla \phi(2^{-j} q_1,2^{-j} q_2) \|_\infty \leq 2^{-j'} \}.
\]
Then there exists $C > 0$ such that the cardinality of $\mathcal{K}_{j}(j')$ obeys
\[
| \mathcal{K}_{j}(j') | \leq C \, 2^{j+(j-j')_{+}}
\]
where $( x )_{+} = x$ if $x \geq 0$, and zero otherwise.
\end{lemma}

\begin{proof}

Let $\r = (\x(s)-\x(t))/ \| \x(s)-\x(t) \|$; the gradient of the phase is $\nabla \phi(s,t) = (\dot\x(s) \cdot r, - \dot\x(t) \cdot \r)$. The condition $\| \nabla \phi(s,t) \|_\infty \leq 2^{-j'}$, i.e.  
\[
| \dot\x(s) \cdot \r | \leq 2^{-j'} \qquad \mbox{and} \qquad  | \dot\x(t) \cdot \r | \leq 2^{-j'},
\]
is for large $j'$ an almost-perpendicularity condition between tangent vectors to the curve $\Omega_\alpha$ and the chord joigning $\x(s)$ and $\x(t)$.

Now fix $s = 2^{-j} q_1$, and let $\n(s)$ be either normal vector to $\Omega_\alpha$ at $\x(s)$. Let $\theta$ be the angle between $\r$ and $\n(s)$, such that $| \dot\x(s) \cdot \r | = \| \dot\x(s) \| \, | \sin \theta |$. Since the parametrization is nondegenerate, the first condition $| \dot\x(s) \cdot \r | \leq 2^{-j'}$ implies $\theta \leq C \, 2^{-j'}$ for some adequately large $C > 0$.

Consider therefore a cone $\Gamma$ with apex at $\x(s)$, axis $\n(s)$, and opening $\theta \leq C \, 2^{-j'}$. The second condition $| \dot\x(t) \cdot \r | \leq 2^{-j'}$ is satisfied only if the curve $\Omega_\alpha$ intersects a chord inside the cone at a near-right angle, and as a consequence, every chord inside the cone at a near-right angle, differing from $\pi/2$ by a $O(2^{-j'})$. Because $\Omega_\alpha$ has finite length, bounded curvature, and obeys the geometric regularity property (\ref{eq:geomreg}), there can only be a finite number of such intersections. The total length of $\Omega_\alpha \cap \Gamma$ is therefore a $O(2^{-j'})$.

Since the points $\x(2^{-j} q_2)$ are a distance $C \, 2^{-j}$ apart from each other, there are at most $O(\max (1, 2^{j-j'}))$ points indexed by $q_2$ that obey the two almost-orthogonality conditions, which can be written as $O(2^{(j-j')_+})$. Since $q_1$ takes on $O(2^j)$ values, the total number of couples $(q_1, q_2)$ obeying the conditions is $O(2^j \, 2^{(j-j')_+})$.

\end{proof}

We will also need the observation that near-stationary-phase points can only occur far away from the diagonal.

\begin{lemma}\label{teo:farfield}
As before, let $d(s,t) = | e^{2 \pi i s} - e^{2 \pi i t} |$. There exist two constants $C_1, C_2 >0$ such that, if $d(s,t) \leq C_1$, then
\[
\| \nabla \phi(s,t) \|_\infty \geq C_2.
\]
\end{lemma}
\begin{proof}
As previously,
\[
\| \nabla \phi(s,t) \|_\infty = \min (| \dot\x(s) \cdot \r |, | \dot\x(t) \cdot \r |),
\]
and we write $| \dot\x(s) \cdot \r |$ as $\| \dot\x(s) \| \, |\cos(\theta_s)|$, where $\theta_s$ is the angle between the chord $(\x(s),\x(t))$ and the tangent vector $\dot\x(s)$. This angle obeys $|\theta_s| \lesssim d(s,t)|$, hence the cosine factor is greater than 1/2 as long as $d(s,t) \leq C_1$ for some adequate $C_1$. The factor $\| \dot\x(s) \|$ is also bounded away from zero by regularity of the parametrization. The same argument can be made for $| \dot\x(t) \cdot \r |$.
\end{proof}

\subsection{Nondiagonal kernel fragments: decay of individual coefficients}\label{sec:nondiag}

Within nondiagonal squares, $d(s,t) \gtrsim 1/k$ and $k \phi(s,t) \gtrsim 1$, so Lemma \ref{teo:hankel1} asserts that $K$ can be written as
\[
K(s,t) = e^{ik \phi(s,t)} a(k \phi(s,t), s, t),
\]
where $\phi(s,t) = \| \mathbf{x}(s) - \mathbf{x}(t) \|$ and the dependence of $a$ on $k$ is mild in comparison to that of $e^{ik \phi}$; 
\[
| \frac{d^{n}}{d \phi^n} a(k \phi(s,t),s,t)| \leq C_n \frac{1}{\sqrt{k \phi(s,t)}} \phi(s,t)^{-n}.
\]
The presence of additional arguments $s$ and $t$ is needed to account for factors such as the Jacobian $\| \x'(t) \|$; all the derivatives of these factors are $O(1)$ by assumption. Therefore, the chain rule yields
\begin{equation}\label{eq:bound-ak}
| \frac{d^{\alpha_1}}{d s^{\alpha_1}} \frac{d^{\alpha_2}}{d t^{\alpha_2}} a(k \phi(s,t),s,t)| \leq C_\alpha \frac{1}{\sqrt{k \phi(s,t)}} \phi(s,t)^{-|\alpha|}, \qquad\qquad \phi(s,t) \lesssim 1.
\end{equation}

%In this section we will assume that the nondiagonal dyadic squares at scale $j$ are included in the nonsingular region $d(s,t) \geq C/k$, which is equivalent to the condition $j \leq \log_2 k + C$ for some other constant $C$. This restriction encompasses the physically relevant scales, including $J = \frac{1}{2} \log_2 k$ that corresponds to the wave number $k$. The very-high-frequency regime where $j \gtrsim \log_2 k$ will still be considered out of mathematical conscience in Section \ref{sec:veryhigh}.

Fix $j > 0$ and $Q \in \mathcal{Q}_j$. We seek a good bound on
\[
\< K_Q, \vf_{j,\m,\n} \> = \int_{3Q} w_Q(s,t) \, a(k\phi(s,t),s,t) \, e^{ik \phi(s,t)} e^{- i 2^j \m \cdot (s,t)} 
\]
\[
\qquad \qquad \times \, 2^{j} \, \vf_{(j,\m)}(2^j s - n_1, 2^j t - n_2) \, ds \, dt,
\]
where $\vf_{(j,\m)}$ has been introduced in equation (\ref{eq:vfjm1}). Without loss of generality, we perform a translation to choose the coordinates $s$ and $t$ such that $x_Q = 0$ and $w_Q(s,t) = w(2^j s, 2^j t)$.

A first bound estimating the decay in $\n$ can be obtained by using 1) the almost-exponential decay (\ref{eq:vfjm2}) for $\vf_{(j,\m)}$, 2) the estimate $\| w_Q \|_{L^1} \lesssim 2^{-2j}$ that follows from $| 3Q | \lesssim 2^{-2j}$, and 3) an $L^\infty$ bound for the rest of the integrand, disregarding the oscillations. The result is
\[
| \< K_Q, \vf_{j,\m,\n} \> | \leq C_M \, 2^{-j} \, \sup_{(s,t) \in 3Q} \left[ (k \phi(s,t))^{-1/2} \right] \,(1 + \| \n \|)^{-M}, \qquad \forall \, M > 0.
\]
The size of the first-order Taylor remainder of $k \phi(s,t)$ over $3Q$ is $O(2^{j})$ times smaller that the value of $k \phi(x_Q)$ itself, so we may evaluate $\phi$ at $x_Q$ at the expense of a multiplicative constant in the estimate. We get
\begin{equation}\label{eq:decay-n}
| \< K_Q, \vf_{j,\m,\n} \> | \leq C_M \, 2^{-j} \, (k \phi(x_Q))^{-1/2} \,(1 + \| \n \|)^{-M}, \qquad \forall \, M > 0.
\end{equation}

Capturing the decay in $\m$, however, requires integrations by parts. Heuristically, the objective is to show that the wave atom coefficients decay almost exponentially in $\m$, with a length scale of 1 in all directions (in units of $\m$), independently of $j$---at least in the representative case $j \simeq \frac{1}{2} \log_2 k$. To this end let us introduce the self-adjoint differential operator
\[
L = \frac{I - \beta \Delta_{(s,t)} - i \beta k \, ( \Delta \phi(s,t) )}{1 + \beta \| k \nabla \phi(s,t) - 2^j \m \|^2},
\]
with
\[
\beta = \frac{1}{\max ( 2^{-2j} k^2, 2^{2j} )}.
\]
We see that $L$ leaves the exponential $\exp{i [k \nabla \phi(s,t) - 2^j \m \cdot (s,t)]}$ unchanged, hence we introduce $M$ copies of $L$, and integrate by parts in $s$ and $t$ to pass the differentiations to the non-oscillatory factors. The scaling parameter $\beta$ has been chosen such that the repeated action of $L$ on the rest of the integrand introduces powers of $1/(1 + \beta \| k \nabla \phi(s,t) - 2^j \m \|^2)$, but otherwise only worsens the bound by a constant independent of $\mu = (j, \m, \n)$. Indeed, $\beta \leq 2^{-2j}$, and
\begin{itemize}
\item the action of each derivative on $w_Q$ or $\vf_{(j,\m)}$ produces a factor $2^j$ balanced by $\sqrt{\beta}$;
\item the action of each derivative on $a$ produces a factor $1/\phi(s,t)$, which by equation (\ref{eq:geomreg}) is comparable to $1/d(s,t)$. Since we are in the presence of nondiagonal squares, $1/d(s,t) \lesssim \min( 2^{j}, k) \leq 2^{j}$. Again, each derivative produces a factor $2^j$, which is balanced by $\sqrt{\beta}$. Note that the leading factor $1/\sqrt{k\phi}$ in the bound (\ref{eq:bound-ak}) is harmless since it is carried through the differentiations.
\end{itemize}
It is then tedious but straightforward to combine these observations and conclude that, for all $M > 0$,
%\[
%| L^M \left[ w(2^j s, 2^j t) \, a_{(k)}(s,t) \, \vf_{(j,\m)}(2^j s - n_1, 2^j t - n_2) \right] | \leq \frac{C_M}{( 1 + \beta \| k \nabla \phi(s,t) - 2^j \m \|^2)^M}.
%\]
\begin{align*}
&| L^M \left[ w(2^j s, 2^j t) \, a(k\phi(s,t),s,t) \, \vf_{(j,\m)}(2^j s - n_1, 2^j t - n_2) \right] | \leq \\
&\qquad\qquad\qquad C_M \; \frac{1}{\sqrt{k \phi(s,t)}} \; \frac{1}{( 1 + \beta \| k \nabla \phi(s,t) - 2^j \m \|^2)^M}.
\end{align*}

Since $L$ is a differential operator, the support of the integrand remains $3Q$ regardless of $M$, hence we still get a factor $|3Q| \sim 2^{-2j}$ from the integral over $s$ and $t$. With the $L^2$ normalization factor $2^j$ coming from equation (\ref{eq:vfjm1}), the overall dependence on scale is $2^{-j}$. The resulting bound is
\begin{equation}\label{eq:decay-m}
| \< K_Q, \vf_{j,\m,\n} \> | \leq C_M \, 2^{-j} \, \sup_{(s,t) \in 3Q} \left[ (k\phi(s,t))^{-1/2} (1 + \beta \| k \nabla \phi(s,t) - 2^j \m \|^2)^{-M} \right],\qquad \forall \, M > 0.
\end{equation}
The second factor inside the square brackets can be written as
\[
\left(1 + \| \frac{k 2^{-j} \nabla \phi(s,t) - \m}{2^{-j} \beta^{-1/2}} \|^2 \right)^{-M},
\]
showing that in $m$-space, it is a fast-decaying bump centered at $k 2^{-j} \nabla \phi(s,t)$ and of characteristic width $2^{-j} \beta^{-1/2}$. Over the set $3Q$, we have the estimate $ | k 2^{-j} \nabla \phi(s,t) - k 2^{-j} \nabla \phi(x_Q) |  = O(k 2^{-2j})$. The quantity $k 2^{-2j}$ is in all cases less than the length scale $2^{-j} \beta^{-1/2}$ (which is why we could not simply have taken $\beta = 2^{-2j}$), so we may replace $\nabla \phi(s,t)$ by $\nabla \phi(x_Q)$ in the expression of the bump, at the expense of a multiplicative constant depending only on $M$. 

We have also seen earlier that $(k \phi(s,t))^{-1/2}$ can safely be replaced by $(k \phi(x_Q))^{-1/2}$ in the region $d(s,t) \gtrsim 1$, at the expense of another multiplicative constant. With these observations, we can take the geometric mean of (\ref{eq:decay-n}) and (\ref{eq:decay-m}) and obtain the central bound
\begin{equation}\label{eq:decay-nondiag}
| \< K_Q, \vf_{j,\m,\n} \> | \leq C_M \, 2^{-j} \, (k \phi(x_Q))^{-1/2} \, (1 + \beta \| k \nabla \phi(x_Q) - 2^j \m \|^2)^{-M} \, (1 + \| \n \|)^{-M},
\end{equation}
for all $M > 0$.

\subsection{Nondiagonal kernel fragments: $\ell_p$ summation}

The expression just obtained can be used to show $\ell_p$ summability and verify proper growth as a function of $k$. Only the case $p \leq 1$ is interesting and treated in this section. We tackle the different sums in the right-hand-side of (\ref{eq:ptriangle-KQ}) in the order as written, from right to left.
\begin{itemize}
\item The sum over $\n$ is readily seen to contribute a multiplicative constant independent of the other parameters $j, Q$, and $\m$.
\item Consider the sum over $\m$, and pull out the factor $2^{-jp} (k \phi(x_Q))^{-p/2}$. (We will not worry about this factor until we treat the sum over $Q$.) The range of values for $\m$ is an annulus $\| \m \|_\infty \asymp 2^{j}$, so we can compare the sum over $\m$ to the integral
\[
\int_{C_j} (1 + \beta \| k \nabla \phi (x_Q) - 2^j \x \|^2)^{-Mp} d \x, 
\]
where $C_j = \{ \x \in \R^2: C_1 2^j \leq \| \x \|_\infty \leq C_2 2^j \}$ for some $C_1, C_2 > 0$. In what follows, take $M$ sufficiently large so that, say, $Mp \geq 5$. Two cases need to be considered, corresponding to $2^{2j} \leq k$ (large scales), and $2^{2j} > k$ (small scales).

\begin{itemize}
\item If $2^{2j} \leq k$, then $\beta = 2^{2j} k^{-2}$. It will be sufficient to consider only the upper bound for $\| x \|_\infty$, whence
we have the bound
\[
\int_{\| \x \|_\infty \leq C 2^j} \left( 1 + \| \frac{ 2^{-j} k  \nabla \phi (x_Q) - \x}{2^{-2j} k} \|^2 \right)^{-Mp} d \x.
\]
With Lemma \ref{teo:scaledefect} in mind, we introduce the \emph{scale defect} $j_Q'$ as the unique integer such that
\[
\frac{1}{2} 2^{-j'_Q} < \| \nabla \phi(x_Q) \| \leq 2^{-j'_Q}.
\]
The integrand is a bump that essentially lies outside of the region of integration as soon as $k 2^{-(j+j'_Q)} \gtrsim 2^j$. 

More precisely, observe that 
\[
\sup_{\x: \| \x \|_\infty \leq C 2^j} \left( 1 + \| \frac{ 2^{-j} k  \nabla \phi (x_Q) - \x}{2^{-2j} k} \|^2 \right)^{-Mp} \leq C \left( 1+ 2^j (2^{-j'_Q} - C' k^{-1} 2^{2j})_+ \right)^{-2Mp}
\]
hence the integral is bounded by a first expression,
\begin{equation}\label{eq:bumpbound1}
C \, 2^{2j} \left( 1+ 2^j (2^{-j'_Q} - C k^{-1} 2^{2j})_+ \right)^{-2Mp}.
\end{equation}

A second bound can be obtained by letting $\x' = \x - 2^{-j} k  \nabla \phi (x_Q)$ and extending the region of integration to the complement of a square in $\x'$, of the form
\[
\| \x' \|_\infty \geq s_{jQ} = \left( \frac{1}{2} k 2^{-(j + j_Q')} - C 2^j \right).
\]
If $s_{jQ} \leq k 2^{-2j}$, we might as well put it to zero and obtain the bound $C (k^2 2^{-2j})^2$ for the integral. If $s_{jQ} > k 2^{-2j}$, the integrand can be made homogeneous in $\x$ and the integral bounded by
\[
\int_{r > s_{jQ}} \left( \frac{r}{k 2^{-2j}} \right)^{-2Mp} \, r dr \lesssim (s_{jQ})^2 \, \left( \frac{s_{jQ}}{k 2^{-2j}} \right)^{-2Mp} \leq (k 2^{-2j})^2 \, \left( \frac{s_{jQ}}{k 2^{-2j}} \right)^{-2Mp+2}
\]
Now uniformly over $s_{jQ}$, the resulting bound is
\begin{equation}\label{eq:bumpbound2}
C \, (k 2^{-2j})^2 \, \left( 1 + 2^j (2^{-j'_Q} - k^{-1} 2^{2j})_+ \right)^{-2Mp+2}.
\end{equation}
The minimum of (\ref{eq:bumpbound1}) and (\ref{eq:bumpbound2}) is
\begin{equation}\label{eq:boundm1}
C \, \min \left( 2^{2j}, k^2 2^{-4j} \right) \, \left( 1 + 2^j (2^{-j'_Q} - k^{-1} 2^{2j})_+ \right)^{-2Mp+2}.
\end{equation}

\item If $k \leq 2^{2j}$, then $\beta = 2^{-2j}$. This time we will only consider the lower bound for $\| \x \|_\infty$, and write
\[
\int_{\| \x \|_\infty \geq C 2^j} (1 + \| 2^{-j} k \nabla \phi (x_Q) -  \x \|^2)^{-Mp} d \x.
\]
Since $\| \nabla \phi(x_Q) \|$ is a $O(1)$ and $2^{-j} k \leq 2^j$, there exists a value $j^* \leq \frac{1}{2} \log_2 k + C$ such that for all $j \geq j^*$, the center of the bump is inside the square $\| \x \|_\infty \leq C 2^j$ (the constant $C$ changes from expression to expression.) When this occurs, we can let $\x' = \x - 2^{-j} k  \nabla \phi (x_Q)$ as before and consider the integral outside of a smaller square and bound
\begin{equation}\label{eq:boundm2}
\int_{\| \x \|_\infty \geq C 2^j} (1 + \| \x \|^2)^{-Mp} d \x \leq 2^{-2j(Mp-2)}, \qquad j \geq j^*.
\end{equation}
For the few values of $j$ such that $\frac{1}{2} \log_2 k \leq j \leq j^*$, we recover the previous estimate, namely 
$C \, (k 2^{-2j})^2$, which is a $O(1)$.

\end{itemize}

\item Consider now the sum over $Q$, and recall that the bounds just obtained need to be multiplied by $2^{-jp} (k \phi(x_Q))^{-p/2}$. Pull out the factor $2^{-jp}$ one more time. Again, we need to separately consider $2^{2j} \leq k$ (large scales) and $2^{2j} \geq k$ (small scales). For small scales,  the bound (\ref{eq:boundm2}) is uniform in $Q$, hence the sum over $Q \in \mathcal{Q}_j$ simply contributes a factor $2^{2j}$. 

For large scales, the strategy is to split the sum over $Q$ into a \emph{near-field contribution}, for which $d(s,t) \leq C_1$ in the sense of Lemma \ref{teo:farfield}, and a \emph{far-field contribution}, for which stationary phase points must be handled adequately. The terms in the far-field sum is then further broken down into groups corresponding to a same value of the scale defect, which Lemma \ref{teo:scaledefect} helps identify. Schematically,
\[
\sum_{Q \in \mathcal{Q}_j} \; = \; \sum_{Q \in \mbox{ near-field}} \; + \; \sum_{j' > 0}  \; \; \left[ \sum_{Q: \mbox{ scale defect} = j'} \right].
\]

Consider the two regions separately.
\begin{itemize}
\item \emph{Near-field}. In this region, $\phi(x_Q)$ may be as small as $1/k$, hence we estimate $(k \phi(x_Q))^{-p/2} \lesssim 1$. By Lemma \ref{teo:farfield}, the scale defect $j'_Q$ is bounded by a constant, hence the bound (\ref{eq:boundm1}) becomes
\begin{equation}\label{eq:interm-bound}
C \, \min \left( 2^{2j}, k^2 2^{-4j} \right) \, \left( 1 + 2^j (1- k^{-1} 2^{2j})_+ \right)^{-2Mp+2}.
\end{equation}

We claim that this quantity is always less than a constant independent of $j$ and $k$. Indeed, if $j$ is so large that $2^{2j} k^{-1} \geq 1/2$, then $k^2 2^{-4j} \leq 4$, and it suffices to use the trivial minoration $1 + 2^j (1- k^{-1} 2^{2j})_+ \geq 1$. If on the other hand  $2^{2j} k^{-1} < 1/2$, then we have $1 + 2^j (1- k^{-1} 2^{2j})_+ \geq \frac{1}{2} 2^j$, which implies that (\ref{eq:interm-bound}) is bounded by
\[
C \, \min \left( 2^{2j}, k^2 2^{-4j} \right) \, 2^{-2j(Mp-2)} \leq C \, 2^{-2j(-1 + Mp - 2)} \leq 1,
\]
because we chose $Mp \geq 5$. The sum over $Q$ then contributes a factor proportional to the number of nondiagonal squares, i.e., $2^{2j}$.

\item \emph{Far-field}. The leading factor $(k \phi(x_Q))^{-p/2}$ now contributes a factor $k^{-p/2}$, since $\phi(x_Q) \geq C$ in the far field. For the sum over $Q$, we use equation (\ref{eq:boundm1}) one more time and write
\[
C \, \min \left( 2^{2j}, k^2 2^{-4j} \right) \; \sum_{Q \in \mbox{ far-field}} \; \left( 1 + 2^j (2^{-j'_Q} - k^{-1} 2^{2j})_+ \right)^{-Mp+2}.
\]
For each $Q$, find the closest integer $j' \leq j$ to $j'_Q$. As long as $j' < j$, Lemma \ref{teo:scaledefect} asserts that the number of terms comparable to $\left( 1 + 2^j (2^{-j'} - k^{-1} 2^{2j})_+ \right)^{-Mp+2}$ is a $O(2^{2 j -j'})$. The endpoint $j' = j$ receives the contribution of arbitrary large $j'_Q$, meaning terms that can be as large as a $O(1)$; however by Lemma \ref{teo:scaledefect} there can only be $O(2^j)$ such terms. After indexing terms by $j'$ in place of $Q$, we get the bound
\[
C \, \min \left( 2^{2j}, k^2 2^{-4j} \right) \left[ \, 2^j +  \sum_{-C \leq j' < j} 2^{2j - j'} \, \left( 1 + 2^j (2^{-j'} - k^{-1} 2^{2j})_+ \right)^{-Mp+2} \right].
\]
It is easy to see that the summand peaks for $j'$ near $j_0 = \min (j,  -2j + \log_2 k)$; it decreases geometrically for $j \leq j_0$ because of the factor in brackets, and decreases geometrically for $j \geq j_0$ because of the factor $2^{2j - j'}$. The result is a bound
\[
C \, \min \left( 2^{2j}, k^2 2^{-4j} \right) \,  \max (2^{4j} k^{-1}, 2^j)  = C \, \min (k, 2^{3j}).
\]

\end{itemize}

\item What remains after gathering the various bounds is a constant $C_{M,p}$ times

\begin{align*}
&\sum_{j \leq \frac{1}{2} \log_2 k + C} 2^{-jp} \, 2^{2j} \qquad\qquad\qquad\qquad\quad\; \mbox{(near-field)} \\
+ &\sum_{j \leq \frac{1}{2} \log_2 k + C} k^{-p/2} 2^{-jp} \min(k, 2^{3j}) \qquad\qquad \mbox{(far-field)} \\
+ &\sum_{j > \frac{1}{2} \log_2 k + C} 2^{2j} 2^{-jp} 2^{-2j(Mp - 2)}. \qquad\qquad \mbox{(small scales)}
\end{align*}

The near-field contribution sums up to $C_p \, k^{1-p/2}$ as soon as $p < 2$. The far field contribution is bounded by
\begin{align*}
&C_p \, k^{-p/2} \left[ \sum_{j \leq \frac{1}{3} \log_2 k} 2^{j(3-p)} + k \sum_{j > \frac{1}{3} \log_2 k} 2^{-jp} \right] \\
&\qquad \leq C_p \, k^{-p/2} \, k^{1-p/3} \leq C_p \, k^{1-p/2}.
\end{align*}
With the choice $Mp \geq 5$, the contribution of ``small scales'' is negligible in contrast to the first two terms. $O(k^{1-p/2})$ is the desired growth rate in $k$, compatible with equation (\ref{eq:Kmu-ellp1}). This concludes the part of the proof related to nondiagonal squares.

\end{itemize}

\subsection{Diagonal kernel fragments: decay of individual coefficients}

It is now assumed that the dyadic square $Q$ overlaps with the diagonal strip $S = \{ (s,t): d(s,t) \lesssim 1/k \}$. Because of the singularity of the kernel at $s=t$, the integrations by parts cannot proceed as before. Inside $S$, the smoothness of the kernel is governed by the case $kx \lesssim 1$ in Lemma \ref{teo:hankel1}.

Further complications arise, depending on the value of the scale $j$. 
\begin{itemize}
\item At scales $j \leq \log_2 k + C$, a square $Q$ intersecting with $S$ is not entirely contained in $S$; in fact, a large portion of it lies in the nondiagonal portion $d(s,t) \gtrsim 1/k$. We call this portion (two triangles) the \emph{off-strip contribution}. There are $O(2^j)$ such triangles. 
\item At scales $j \geq \log_2 k + C$, some squares may be contained inside the strip $S$ without intersecting the diagonal $s=t$. These squares make up the \emph{regular on-strip contribution}; there are $O(2^{2j} k^{-1})$ such squares.
\item The remaining $O(2^j)$ squares or portions thereof, overlapping with the diagonal $s=t$, make up the \emph{singular on-strip contribution}.
\end{itemize}

In order to smoothly cut off the strip $S$ from dyadic squares, introduce $\sigma = s-t$ (defined modulo 1 in $[0,1]$) and $\tau = s+t$. By symmetry of the problem, one can consider the triangle $\{(s,t) \in Q: s \geq t \}$ and still call it $Q$, without loss of generality. We can therefore focus on $\sigma > 0$. If we properly select the coset relative to the modulo operation, we can also take $(\sigma, \tau)$ to smoothly parametrize the triangle $Q$. With these choices, the diagonal strip is $S = \{ (\sigma, \tau): \sigma \lesssim 1/k \}.$ Consider now a smooth indicator $\rho(k \sigma)$ where $\rho$ is a $C^\infty$ positive function obeying
\[
\rho(x) = \left\{ \begin{array}{ll}
         1 & \mbox{if $0 \leq x \leq 1$};\\
         0 & \mbox{if $x \geq 2$}.\end{array} \right. 
\]
Multiplying the integrand in (\ref{eq:coeff-diag}) by $\rho(k \sigma)$ gives the on-strip contribution (regular and singular); multiplying it by  $1 - \rho(k \sigma)$ gives the off-strip contribution. These cases are treated separately.

\begin{itemize}
\item \emph{Off-strip contribution.} Over the off-strip region we have $k \sigma \gtrsim 1$ hence $k \phi(s,t) \gtrsim 1$ by 
equation (\ref{eq:geomreg}), so that the case $kx \gtrsim 1$ of Lemma \ref{teo:hankel1} applies there. The analysis of the coefficient decay in $\n$ is the same as in the previous section, so we omit it here.

As far as analysis of the decay in $\m$, we are back in the setting of the analysis of Section \ref{sec:nondiag}, except for the factor $1-\rho(k\sigma)$ that prevents the same scheme of integrations by parts in $\sigma$. Each derivative of $\rho(k\sigma)$ would produce an unacceptably large factor $k$. The smoothness in $\tau$ is however unaffected, which permits to carry over the analysis of Section \ref{sec:nondiag} with integrations by parts in $\tau$ only. The direction of increasing $\tau$ in the $\m$ plane is $\mathbf{e}_\tau = (1,1)$, to which corresponds the decomposition $2 \, \m \cdot (s,t) = (m_1 + m_2)\tau + (m_1 - m_2) \sigma$. Since $\nabla \phi(0) \equiv \lim_{\sigma \to 0^+} \nabla \phi$ points in the direction of $\sigma$, we have $\mathbf{e}_\tau \cdot \nabla \phi(0) = 0$, and we can use
\[
L_0 = \frac{I - \beta \frac{\pd^2}{\pd \tau^2} - i \beta k \, ( \frac{\pd^2 \phi}{\pd \tau^2} )}{1 + \beta 2^{2j} \left( \frac{m_1 + m_2}{2} \right)^2},
\]
to generate the repeated integrations by parts, where $\beta$ is the same as previously. Notice that the amplitude is uniformly bounded, since $kx \gtrsim 1$ in the off-strip region. The result is a bound that involves $m_1 + m_2$ only. With the contribution of the decay in $\n$, the off-strip coefficient estimate is
\begin{equation}\label{eq:offstrip}
| \< K_Q, (1-\rho(k\sigma)) \vf_{j,\m,\n} \> | \leq C_M \, 2^{-j} \, ( 1+ \beta 2^{2j} ( m_1 + m_2 )^2)^{-M} (1 + \| \n \|)^{-M},
\end{equation}
for all $M > 0$, and only for scales obeying $2^{2j} \lesssim k$.

\item \emph{On-strip contribution: amplitude estimate.} We now take $k \sigma \lesssim 1$. The case $kx \lesssim 1$ of Lemma \ref{teo:hankel1} allows to write
\[
K(s,t) = e^{i k \phi(s,t)} a(k \phi(s,t),s,t),
\]
where now the amplitude's smoothness is
\begin{equation}\label{eq:dadphi}
| \frac{d^n a}{d \phi^n}(k \phi,s,t)| \leq C_n \, \phi^{-n}, \qquad \qquad \phi \lesssim 1.
\end{equation}
The partial derivatives of $a$ with respect to the arguments $s$ and $t$ are $O(1)$ and well within the above bound as long as $\phi(s,t) \lesssim 1$. To compute the \emph{total} derivatives with respect to $s$ and $t$, however, it is necessary to contrast smoothness along and across the oscillations, by means of the coordinates $\sigma$ and $\tau$. The value of $\phi(s,t)$ is comparable to the circle distance $d(\sigma,0)$, namely
\[
D d(\sigma,0) \leq \phi \left( \frac{\sigma+\tau}{2} , \frac{\tau-\sigma}{2} \right) \leq \tilde{D} d(\sigma,0).
\]
The first inequality is exactly equation (\ref{eq:geomreg}), the last inequality follows from a Taylor expansion. Since we only consider $\sigma > 0$, we write this property as $\phi \asymp \sigma$. A careful analysis of Taylor remainders shows that the same estimate estimate is true for the $\tau$ derivatives,
\begin{equation}\label{eq:dphidtau}
| \frac{d^n \phi}{d \tau^n} \left( \frac{\sigma+\tau}{2}, \frac{\tau-\sigma}{2}\right) | \leq C_n \, \sigma, \qquad\qquad \sigma \ne 0,
\end{equation}
while the $\sigma$ derivatives do not yield any gain: 
\begin{equation}\label{eq:dphidsigma}
| \frac{d^n \phi}{d \sigma^n} \left( \frac{\sigma+\tau}{2}, \frac{\tau-\sigma}{2}\right) | \leq C_n, \qquad\qquad \sigma \ne 0.
\end{equation}
The action of the successive $\tau$ derivatives on $a$ through its $\phi$ dependence can be understood from the higher-order analogue of the chain rule, known as the combinatorial Fa\`{a} di Bruno formula:
\[
\left( \frac{d}{d \tau} \right)^n a\left( k \phi \left( \frac{\sigma+\tau}{2} , \frac{\tau-\sigma}{2} \right),s_0,t_0 \right) = \sum_{\pi \in \Pi} \left( \frac{d^{| \pi |} a}{d \phi^{| \pi |}} \right) \cdot \prod_{B \in \pi} k \frac{d^{|B|} \phi}{d \tau^{|B|}}.
\]
In this formula, $\Pi$ is the set of all partitions $\pi$ of $\{ 1, \ldots, n \}$; $|\pi|$ denotes the number of blocks in the partition $\pi$; these blocks are indexed as $B \in \pi$; and $|B|$ denotes the size of the block $B$. Since there are $| \pi |$ factors in the product over $B$, equation (\ref{eq:dphidtau}) reveals that the derivatives of $\phi$ yield a factor $C_n \, \sigma^{|\pi|}$. On the other hand, by equation (\ref{eq:dadphi}), each $\phi$-differentiation of $a$ introduces an inverse power of $\sigma$. The order of the derivative is $| \pi |$, for a contribution of $\sigma^{- |\pi|}$ that exactly cancels the $\sigma^{|\pi|}$ coming from the derivatives of $\phi$. 

This analysis only concerns the dependence of $a$ on $\tau$ via $\phi$. It is easy to apply the multivariable chain rule to see that the depedence of $a$ on $\tau$ via its second and third arguments ($s$ and $t$) does not change the conclusion that any number $n$ of $\tau$ derivatives, $n \geq 1$, keep the amplitude bounded, with bound independent of $j$ and $k$ (but not $n$, of course):
\begin{equation}\label{eq:dadtau}
| \frac{d^n a}{d \tau^n}\left( k \phi\left( \frac{\sigma+\tau}{2} , \frac{\tau-\sigma}{2} \right), \frac{\sigma+\tau}{2} , \frac{\tau-\sigma}{2} \right) | \leq C_n, \qquad\qquad n \geq 1, \; \sigma \ne 0.
\end{equation}

There are no factors to gain in the $\sigma$ derivatives of the phase, hence the same analysis yields
\begin{equation}\label{eq:dadsigma}
| \frac{d^n a}{d \sigma^n}\left( k \phi\left( \frac{\sigma+\tau}{2} , \frac{\tau-\sigma}{2} \right), \frac{\sigma+\tau}{2} , \frac{\tau-\sigma}{2} \right) | \leq C_n \, \phi^{-n}, \qquad\qquad \phi \lesssim 1, \; n \geq 1, \; \sigma \ne 0.
\end{equation}

We are now equipped to study the coefficient
\[
\< K_Q, \vf_{j,\m,\n} \> = \int_{3Q} w(2^j s,2^j t) e^{i k \phi(s,t)} a(k\phi(s,t),s,t) \, e^{- i 2^j \m \cdot (s,t)}
\]
\begin{equation}\label{eq:coeff-diag}
\qquad \qquad \qquad \times \, 2^{j} \, \vf_{(j,\m)}(2^j s - n_1, 2^j t - n_2) \, ds \, dt.
\end{equation}

%*******

%First, consider the decay in $\n$. It suffices to differentiate $a$ once in $\tau$ to get rid of the logarithmic singularity and obtain a bounded integrand. Taking a $\tau$-antiderivative of the rest of the integrand is harmless and does not compromise the decay of $\vf_{(j,\m,\n)}$ in $\n$. The analysis is the same as previously and yields the same estimate as in equation (\ref{eq:decay-n}).

%********** (do this right)

%Note that the same estimate of decay in $\n$ would hold, had we multiplied the integrand by a bounded function as we will do in the sequel.

\item \emph{Regular on-strip contribution}

For regular on-strip squares, i.e., those squares at very small scales $j \geq \log_2 k + C$ that intersect with the strip $S$ but not with the diagonal $s=t$, the decay in $\m$ and $\n$ is obtained by a simple argument of integration by parts. Introduce copies of
\[
L_1 = \frac{I - 2^{-2j} \Delta_{(\sigma,\tau)}}{1 + \| \m \|^2},
\]
and integrate by parts in (\ref{eq:coeff-diag}). Each derivative in $\sigma$ acting on the amplitude $a(k\phi,s,t)$ produces a factor $\phi^{-1} \asymp \sigma^{-1}$. Since $Q$ does not intersect with the diagonal, $\sigma \gtrsim 2^{-j}$ hence $\sigma^{-1} \lesssim 2^j$. This factor is balanced by the choice of scaling in the expression of $L_2$. A fortiori, the derivatives in $\tau$ are governed by a stronger estimate and are therefore under control. The derivative in $\tau$ or $\sigma$ acting on $\vf_{(j,\m,\n)}$ do not compromise its super-algebraic decay, hence we gather the same decay in $\n$ as previously. One complication is however the possible logarithmic growth near $\sigma = 0$ of the amplitude $a$ when it is not differentiated. Consider the intermediate bound
\[
| \< K_Q, \rho(k\sigma) \vf_{j,\m,\n} \> | \leq C_M \, 2^{j} \, \left( 1+ \| \m \|^2 \right)^{-M} \, (1 + \| \n \|)^{-M}  \, \int_{3Q}  | a(k \phi(s,t),s,t) | \, ds dt,
\]
for all $M > 0$. If the amplitude were bounded, then the integral would produce a factor $2^{-2j}$ like in the nondiagonal case. Instead, we claim that the integral factor is bounded by $2^{-j} k^{-1}$. In order to see this, consider  the bound
\[
|a(k \phi,s,t)| \leq C \, (1 + |\log(k \phi(s,t))|),
\]
from Lemma \ref{teo:hankel1}. Since log is increasing and $\phi \asymp \sigma$, there exist $C_1, C_2 > 0$ such that
\[
\log ( C_1 k \sigma) \leq \log( k \phi(s,t)) \leq \log( C_2 k \sigma),
\]
hence
\begin{equation*}
|\log (k\phi(s,t))| \leq C + | \log (k \sigma) |, \qquad \mbox{for some } C > 0.
\end{equation*}
This bound does not depend on $\tau$, and since $(s,t) \in 3Q$, $\tau$ ranges over a set of length $O(2^{-j})$. We therefore obtain the bound
\[
 \int_{3Q}  | a(k \phi(s,t),s,t) | \, ds dt  \leq C \, 2^{-j} \times \int_0^{\frac{1}{Ck}} (C + | \log (k \sigma)|) \, d\sigma \leq C \, 2^{-j} k^{-1}.
\]

The final estimate for the regular on-strip contribution is
\begin{equation}\label{eq:onstrip-reg}
| \< K_Q, \rho(k\sigma) \vf_{j,\m,\n} \> | \leq C_M \, k^{-1} \, \left( 1+ \| \m \|^2 \right)^{-M} \, (1 + \| \n \|)^{-M}, \qquad \mbox{(regular on-strip)}
\end{equation}

\item \emph{Singular on-strip contribution}

Let us now consider a dyadic square $Q$ that intersects with the diagonal $s=t$. Heuristically, we cannot expect that the decay length scale of the wave atom coefficients be independent of $j$ in all directions in $\m$: because $Q$ overlaps with, or is close to the diagonal, the decay in the direction $m_1 - m_2$ is much slower than the decay in the direction $m_1 + m_2$. However, the number of diagonal squares is small enough to restore the overall balance at the level of the $\ell_p$ summability criterion.

To quantify the decay in the $m_1 + m_2$ direction, introduce the self-adjoint operator
\[
L_2 = \frac{I - 2^{-2j} \frac{\pd^2 }{\pd \tau^2}}{1 + \left( \frac{m_1 + m_2}{2} \right)^2}.
\]
It leaves the exponential $e^{- i 2^j (m_1 s + m_2 t)}$ invariant. After integrating by parts, the action of $I - 2^{-2j} \frac{\pd^2 }{\pd \tau^2}$ leaves the bound on the rest of the integrand unchanged, because
\begin{itemize}
\item[(i)] $w(2^j s, 2^j t)$ and $\vf_{(j,\m)}(2^j s - n_1, 2^j t - n_2)$ produce a factor $2^j$ when differentiated;
\item[(ii)] differentiating $\vf_{(j,\m)}(2^j s - n_1, 2^j t - n_2)$ does not compromise its decay in $\n$; and
\item[(iii)] $a(k\phi,s,t)$ has a logarithmic singularity, and otherwise becomes uniformly bounded when differentiated in $\tau$, as we have seen. (The presence of the scaling $2^{-2j}$ in $L_1$ is not even needed here.)
\end{itemize}
The integral $\int_{3Q}  | a(k \phi(s,t),s,t) | \, ds dt$ for the amplitude can be bounded by $C \, 2^{-j}  \, \max (2^{-j}, k^{-1})$ as we arued for the regular on-strip squares (here $3Q$ is not necessarily contained in $S$.) The result is a bound
\begin{equation}\label{eq:onstrip-sing}
| \< K_Q, \rho(k\sigma) \vf_{j,\m,\n} \> | \leq C_M \, \max(2^{-j}, k^{-1})  \, \left( 1+ ( m_1 + m_2 )^2 \right)^{-M} \, (1 + \| \n \| )^{-M}, \,  \qquad \mbox{(singular on-strip)}
\end{equation}
for all $M > 0$.

Finally, the decay in $m_1 - m_2$ for those (singular, on-strip) squares that intersect the diagonal cannot proceed as previously. An analysis of coefficients taken individually would be far from sharp, e.g., would not even reproduce $\ell_2$ summability. The proper reasoning involves a \emph{collective bound} on the $\ell_2$ norm of all the wave atom coefficients at a given scale $j > 0$, which correspond to squares $Q$ that intersect with the diagonal. This reasoning is explained in the next section, and gives the bound
\begin{equation}\label{eq:collective}
\sum_{\m, \n} | \< K_Q, \rho(k\sigma) \vf_{j,\m,\n} \> |^2 \leq C \; j^2 2^{-3j}, \qquad Q \in \mathcal{Q}_j \mbox{ and $Q$ intersects the diagonal.}
\end{equation}

The study of $\ell_p$ summability from all these estimates is then treated in Section \ref{sec:diag-ellp}.

\end{itemize}

\subsection{Diagonal kernel fragments: collective decay properties}\label{sec:collective}

The strategy for obtaining (\ref{eq:collective}) is to compare wave atom coefficients to wavelet coefficients, scale by scale. Estimating individual wavelet coefficients is a much tighter way to capture the sparsity of a log singularity than directly through wave atoms. (Wavelets however are not well-adapted for the overwhelming majority of dyadic squares that correspond to $C^\infty$ oscillations.)

Consider two-dimensional compactly supported Daubechies wavelets with one dilation index $j$, built on the principle of multiresolution analysis \cite{Mal}. They are denoted as
\[
\psi^\varepsilon_{j',\n}(s,t) = 2^{j'} \psi^\varepsilon(2^{j'} s - n_1, 2^{j'} t - n_2),
\]
where $\varepsilon = 1,2,3$ indexes the type of the wavelet (HH, HL or LL). The easiest way to define Meyer wavelets in a square it to periodize them at the edges.

Fix $j \leq 0$, consider a function $f(s,t)$ defined in $[0,1]^2$, and consider its wave atom coefficients at a scale $j$. By Plancherel for wave atoms,  there exists an annulus $A_j = \{ (\xi_1, \xi_2) : C_1 2^{2j} \leq \| \xi \|_\infty \leq C_2 2^{2j} \}$ such that
\[
\sum_{\m,\n} | \< f, \vf_{j,\m,\n} \>|^2 \leq \int_{A_j} | \hat{f}(\xi) |^2 \, d\xi.
\]
By Plancherel for wavelets and the properties of Daubechies wavelets \cite{Mal}, there exists $j_0$ such that this $L^2$ energy is for the most part accounted for by the wavelet coefficients at scales $2j - j_0 \leq j' \leq 2j + j_0$, i.e.,
\begin{equation}\label{eq:Plancherel-wavelets}
\frac{1}{2} \int_{A_j} | \hat{f}(\xi) |^2 \, d\xi \leq \sum_{j' \in [2j-j_0, 2j + j_0]} \sum_{\varepsilon,\n} |\< f, \psi^\varepsilon_{j',\n} \>|^2.
\end{equation}
The last two equations show that, collectively in an $\ell_2$ sense, wave atom coefficients at scale $j$ can be controlled by wavelet coefficients at scales neighboring $2j$.

The relevant range of scales for this analysis is $j \geq \frac{1}{2} \log_2 k$. The on-strip region has length $O(2^{-j})$ and width $O(\min(k^{-1}, 2^{-j}))$. Since each wavelet is supported in a square of size $\sim 2^{-2j}$-by-$2^{-2j}$, the number of wavelets that intersect the strip is $O(2^{j} \times 2^{2j} \min(k^{-1},2^{-j})) = O(\min(2^{3j} k^{-1}, 2^{j}))$. Among those, only $O(2^{j})$ correspond to wavelets intersecting with the diagonal $s = t$. The bound on wavelet coefficient depends on their location with respect to the diagonal:

\begin{itemize}
\item \emph{Non-diagonal wavelets}. The wavelet's wave number is $ \sim 2^{j'} \sim 2^{2j}$ and soon becomes much larger than the local wave number $\sim k$ of the oscillations of the kernel, hence a fast decay in $j' \to \infty$. More precisely, fix $Q \in \mathcal{Q}_j$; the coefficient of interest is
\begin{equation*}
\< K_Q \rho(k\sigma), \psi^\varepsilon_{j',\n} \> = \int_{\mbox{supp } \psi^\varepsilon_{j',\n}} a(k\phi(s,t),s,t) e^{ik \phi(s,t)} w(2^j s, 2^j t) 2^{j'} \psi^\varepsilon(2^{j'} s - n_1, 2^{j'} t - n_2) \, ds dt.
\end{equation*}
Because the wavelet has at least one vanishing moment, one may write it either as
\[
\psi^\varepsilon(2^{j'} s - n_1, 2^{j'} t - n_2) = 2^{-j'} \frac{d \tilde{\psi}^\varepsilon}{d s} (2^{j'} s - n_1, 2^{j'} t - n_2), \qquad \varepsilon = \mbox{HL or HH},
\]
or as
\[
\psi^\varepsilon(2^{j'} s - n_1, 2^{j'} t - n_2) = 2^{-j'} \frac{d \tilde{\psi}^\varepsilon}{d t} (2^{j'} s - n_1, 2^{j'} t - n_2), \qquad \varepsilon = \mbox{LH},
\]
where $\tilde{\psi}^\varepsilon$ has the same support as $\psi^\varepsilon$. After integrating by parts in $s$ or $t$, we can 
\begin{itemize}
\item use the bounds (\ref{eq:dadtau}) and (\ref{eq:dadsigma}) on the amplitude;
\item use the bound $\nabla_{(s,t)} e^{ik\phi} = O(k)$;
\item use $\nabla_{(s,t)} w(2^j s, 2^j t) = O(2^j)$;
\item use $|\mbox{supp } \psi^\varepsilon_{j',\n}| \lesssim 2^{-2j'}$ 
\end{itemize}
to conclude that the coefficient obeys
\[
| \< K_Q \rho(k\sigma), \psi^\varepsilon_{j',\n} \> | \leq C_M \, 2^{-j'} \,  2^{-j'} \left( \max( k, 2^j) + \overline{\phi^{-1}} \right),
\]
where $\overline{\phi^{-1}}$ is a notation for the supremum of $\phi^{-1}$ over the support of the wavelet. If we index by the integer $q \geq 1$ the distance between the center of the support of the wavelet to the diagonal, as $\sqrt{2} q 2^{-j'}$, then $\overline{\phi^{-1}} \asymp q^{-1} 2^{j'}$. The bound above becomes $C_M \, 2^{-j'} (q^{-1} + 2^{-j'} \max(k, 2^j))$.

The sum in the right-hand-side of (\ref{eq:Plancherel-wavelets}) is then estimated as follows. As we saw earlier the length of the strip $S \cap Q$ is $O(2^{-j})$, and its width is $\min(1/k, 2^{-j})$. Since the translation step of wavelets is $2^{-j'} \sim 2^{-2j}$, the translation index $\n$ takes on $2^j \times 2^{2j} \min(1/k, 2^{-j})$ values. Hence
\begin{align*}
\sum_{j' \in [2j-j_0, 2j + j_0]} \sum_{\varepsilon,\n} |\< K_Q \rho(k\sigma), \psi^\varepsilon_{j',\n} \>|^2 &\leq C \sum_{\n} |2^{-2j} (q^{-1} + 2^{-2j} k)|^2 \\
&\leq C 2^j \sum_{q \geq 1} (2^{-4j} q^{-2}) + 2^{3j} \min \left( \frac{1}{k}, 2^{-j} \right) \times 2^{-4j} (2^{-2j} \max(k, 2^{j}))^{2} \\
&\leq C 2^{-3j}.  \qquad \qquad \mbox{(because } j \geq \frac{1}{2} \log_2 k + C \mbox{ )}
\end{align*}

\item \emph{Diagonal wavelets}. For wavelets intersecting the diagonal, it will not be necessary to quantify cancellations. By Lemma \ref{teo:hankel1},
\begin{equation}\label{eq:waveletscoeff}
| \< K_Q \rho(k\sigma), \psi^\varepsilon_{j',\n} \> | \leq C \, \int_{\mbox{supp } \psi^\varepsilon_{j',\n}} (1 + | \log (k \sigma) |) \, 2^{j'} \, | \psi^\varepsilon (2^{j'} s - n_1, 2^{j'} t - n_2) | \, ds dt.
\end{equation}
Without loss of generality we can consider $k \sigma < 1/2$ and write
\[
|\log_2(k \sigma)| = - \log_2(k \sigma) = - \log_2 (2^{j'} \sigma) - \log_2 k + j' \leq - \log_2 (2^{j'} \sigma) + j'.
\]
Since $\log$ is integrable near the origin, and $| \mbox{supp } \psi^\varepsilon_{j',\n} | \lesssim 2^{-2j'}$, the contribution due to $- \log_2 (2^{j'} \sigma)$ is a $O(2^{-j'})$. The contribution of the lone $j'$, on the other hand, is a $O(j' 2^{-j'})$.

There are $O(2^j)$ diagonal wavelets, each with coefficients $O(j' 2^{-j'}) = O(j 2^{-2j})$, hence the sum of their squares in the range $2j - j_0 \leq j' \leq 2j + j_0$ is $O(2^j \times (j2^{-2j})^2) = O(j^2 2^{-3j})$. 

\end{itemize}

As $j \to \infty$, the contribution of diagonal wavelets manifestly dominates that of nondiagonal wavelets, and we have shown that the resulting estimate is (\ref{eq:collective}).

\subsection{Diagonal kernel fragments: $\ell_p$ summation}\label{sec:diag-ellp}

Let us conclude by calculating the growth of $\sum_j \sum_{Q \in \mathcal{Q}_j} \sum_{\mn} | \< K_Q, \vf_{j,\m,\n} \> |^p$ in the parameter $k$, for those dyadic squares that intersect the strip $k \sigma \lesssim 1$. We start by letting $p \leq 1$.

Consider each contribution separately.
\begin{itemize}

\item \emph{Off-strip contribution}. Recall that $j \leq \frac{1}{2} \log_2 k + C$ in this case. The reasoning entirely parallels that of the previous section and we encourage the reader to focus on the discrepancies.  First, the sum over $\n$ yields a harmless constant factor. Second, use (\ref{eq:offstrip}) and drag the factor $2^{-jp}$ out of the sums over $\m$ and $Q$; the former sum is then comparable to the integral
\[
I = \int_{C_j} (1 + | \beta^{1/2} 2^j (x_1 + x_2) |^2 )^{-Mp} dx_1 dx_2,
\]
with $C_j$ an annulus of inner and outer radii proportional to $2^j$. Over this domain, the integrand concentrates near the union of two ``ridges'' of length $\sim 2^j$ and width $\sim \beta^{-1/2} 2^{-j}$, oriented along the anti-diagonal $x_1 = - x_2$. Note that $\beta = 2^{2j} k^{-2}$. The integral $I$ is therefore bounded by a constant times $2^j \, \times (\beta^{-1/2} 2^{-j}) = 2^{-j} k$.  Third, the sum over $Q \in \mathbf{Q}_j$ that intersect with the strip yields a factor $2^{j}$, proportional to the number of diagonal dyadic squares at scale $j$. The remaining sum is bounded by
\[
C_p \sum_{j \leq \frac{1}{2} \log_2 k + C} 2^{j} (2^{-j} k) 2^{-jp} \leq C_p \, k^{1-p/2}.
\]
This is the desired growth rate in $k$.

\item \emph{Regular on-strip contribution}. Here,  $j \geq \frac{1}{2} \log_2 k + C$, therefore $\| \m \|_\infty \geq C \, 2^j \geq C \, \sqrt{k}$. The factor $(1 + \| \m \|_\infty)^{-M}$ in equation (\ref{eq:onstrip-reg}) therefore yields a negative power $k^{-M/2}$ for all $M > 0$, i.e., $k^{-\infty}$, that sums up to a negligible contribution.

\item \emph{Singular on-strip contribution}. As previously, two scale regimes should be considered. When $j \leq \frac{1}{2} \log_2 k$, we can use the bound (\ref{eq:onstrip-sing}). The sum over $\n$ is harmless; the sum over $\m$ produces a factor $2^{j}$ since there are significant  $O(2^{j})$ values of $\m$ on the ridge $|m_1 + m_2| \leq C$ at scale $j$; the sum over $Q$ produces another factor $2^{j}$ since there are $O(2^{j})$ diagonal dyadic squares a scale $j$. The resulting sum over $j$ is then bounded by
\[
C_p \, \sum_{j \leq \frac{1}{2} \log_2 k} 2^{2j} 2^{-jp} \leq C_p \, k^{1-p/2},
\]
which is again the desired growth rate.

If now $j \geq \frac{1}{2} \log_2 k$, we need to invoke the collective decay estimate (\ref{eq:collective}. Fix $j$ and $Q \in \mathcal{Q}_j$ a singular dyadic square. By equation (\ref{eq:onstrip-sing}), values of $m_1 + m_2$ significantly different from zero will give rise to negligible coefficients that sum up to a $o(k)$. More precisely, let $\delta > 0$ be arbitrarily small. Then the wave atom coefficients in the region $|m_1 + m_2| \geq C 2^{\delta j}$ decay sufficiently fast (take $M \gg 1/\delta$) that their total contribution is a $O(k^{-\infty})$ in $\ell_p$. The significant coefficients at scale $j$ are, again, on a ridge of length $O(2^{j})$ and width $O(2^{\delta j})$, for a combined total of $N = O(2^{j(1+\delta)})$ significant coefficients.

We can now relate the $\ell_2$ norm estimate (\ref{eq:collective}) to an $\ell_p$ estimate, $0 < p < 2$, by means of the H\"{o}lder inequality
\begin{equation}\label{eq:Holder}
\left( \sum_{\m,\n} | \< K_Q \rho(k\sigma), \vf_{j,\m,\n} \> |^p \right)^{1/p} \leq \left( \sum_{\m,\n} | \< K_Q \rho(k\sigma), \vf_{j,\m,\n} \> |^2 \right)^{1/2} \, \times \, N^{\frac{1}{p} - \frac{1}{2}},
\end{equation}
where $N = O(2^{j(1+\delta)})$. After simplification, the right-hand side is bounded by $C_{\delta, p} \, j 2^{j(-2 + 1/p + \delta')}$ where $\delta'$ is another arbitrarily small number, namely $\delta' = \delta (\frac{1}{p} - \frac{1}{2})$. This quantity still needs to be summed over $Q$---there are $O(2^j)$ such squares---and then over $j$; but the summation method will depend how $p$ compares to $1$, and accordingly, which of equations (\ref{eq:ptriangle-KQ}) or (\ref{eq:triangle-KQ}) should be used.

If $p \leq 1$, then (\ref{eq:ptriangle-KQ}) should be used, and we obtain
\[
\| K_\mu \|^p_{\ell_p(F)} \leq C_{\delta, p} \, \sum_{j \geq \frac{1}{2} \log_2 k} \, 2^j \left( j \, 2^{j(-2 + \frac{1}{p} + \delta')} \right)^p = C_{\delta, p} \, \sum_{j \geq \frac{1}{2} \log_2 k} j^p \, 2^{j(2-2p+\delta' p)} ,
\]
which always diverges. However if $p \geq 1$, then  (\ref{eq:triangle-KQ}) implies
\[
\| K_\mu \|_{\ell_p(F)} \leq C_{\delta, p} \, \sum_{j \geq \frac{1}{2} \log_2 k} 2^j \, \left( j \, 2^{j(-2 + \frac{1}{p} + \delta')} \right).
\]
For any $1 < p < 2$, the above series is convergent if for instance we choose $\delta' = \frac{1}{2} \left(1 - \frac{1}{p} \right)$, and the result is a $O(1)$, independent of $k$. This part of the singular on-strip contribution falls into the second category identified at the beginning of the proof, namely equation (\ref{eq:Kmu-ellp2}). This concludes the proof in the case when $K$ is the single-layer potential $G_0$.

\end{itemize}

\subsection{Analysis of the double-layer potential}\label{sec:dlp}

The proof of the sparsity result for the double-layer kernel $G_1$ defined in equation (\ref{eq:DLP}) is a simple modification of that for the single-layer kernel $G_0$.

\begin{itemize}
\item \emph{Nondiagonal part.} The smoothness bound for Hankel functions in Lemma \ref{teo:hankel1} exhibits the same rate for all $n \geq 0$ in the case when $x \gtrsim 1/k$. The other factors, functions of $s$ and $t$ which accompany the Hankel factor in formula (\ref{eq:DLP}), have no bearing on the sparsity analysis since they are smooth and do not depend on $k$ but for the leading factor $ik/4$. As a consequence, $G_1$ can still be written as a product
\[
G_1(s,t) = a(k \phi(s,t),s,t) \, e^{i k \phi(s,t)}, \qquad \mbox{(nondiagonal part, $|s-t| \gtrsim 1/k$),}
\]
where the amplitude obeys the same estimate as previously, but for a factor $k$:
\[
| \frac{d^{n}}{d \phi^n} a(k \phi(s,t),s,t)| \leq C_n \, k \,\frac{1}{\sqrt{k \phi(s,t)}} \phi(s,t)^{-n}.
\]

With a number of nonstandard wave atom coefficients $|\Lambda| = O(k \eps^{-1/\infty})$, one could form an approximation of the nondiagonal part of $G_0$ with error $\eps$. In the case of $G_1$, the same number of terms results in an error that we can only bound by $k \eps$. Thus, to make the error less than a specified $\tilde{\eps}$, the number of terms needs to be $O(k^{1+1/\infty} {\tilde{\eps}}^{-1/\infty})$. This justifies the form of the first term in (\ref{eq:cardinality1}).

\item \emph{Diagonal part.} For $x \lesssim 1/k$ the smoothness estimate in Lemma \ref{teo:hankel1} is worse for $H_1^{(1)}$ than for $H_0^{(1)}$, but as is well-known, the dot product
\[
 \frac{\mathbf{x}(s) -\mathbf{x}(t)}{\| \mathbf{x}(s) - \mathbf{x}(t) \|} \cdot n_{\mathbf{x}(t)} \, \|\dot{\mathbf{x}}(t)\|
\]
is small near the diagonal and more than compensates for the growth of $H_1^{(1)}$ there. The precise version of this heuristic is a decomposition
\[
G_1(s,t) = a(k \phi(s,t),s,t) \, e^{i k \phi(s,t)}, \qquad \mbox{(diagonal strip, $|s-t| \lesssim 1/k$),}
\]
where we claim that the amplitude obeys the same estimates as those for $G_0$ in variables $\sigma = s-t$, $\tau = s + t$, namely
\begin{equation}\label{eq:dadtau2}
| \frac{d^n a}{d \tau^n}\left( k \phi\left( \frac{\sigma+\tau}{2} , \frac{\tau-\sigma}{2} \right), \frac{\sigma+\tau}{2} , \frac{\tau-\sigma}{2} \right) | \leq C_n,
\end{equation}
\begin{equation}\label{eq:dadsigma2}
| \frac{d^n a}{d \sigma^n}\left( k \phi\left( \frac{\sigma+\tau}{2} , \frac{\tau-\sigma}{2} \right), \frac{\sigma+\tau}{2} , \frac{\tau-\sigma}{2} \right) | \leq C_n \, \phi^{-n}, \qquad\qquad \phi \lesssim 1,
\end{equation}
for $\sigma \lesssim 1/k$, and this time for every $n \geq 0$ including zero. (Total derivatives in $\sigma$ are defined by keeping $\tau$ fixed, and vice-versa.)

Let us prove (\ref{eq:dadtau2}). As previously, put $r = \frac{\x(s) - \x(t)}{\| \x(s) - \x(t) \|}$. Observe that $n_{\x(t)} \| \dot\x(t) \| = (\dot\x(t))^\bot \cdot r$. Derivatives of $(\dot\x(t))^\bot \cdot r$ in $\sigma$ and $\tau$ are treated by the following lemma, proved in the Appendix.

\begin{lemma}\label{teo:normalder}
For all $n \geq 0$,
\begin{equation}\label{eq:normaldertau}
| \frac{d^n}{d \tau^n} \left[ (\dot\x(t))^\bot \cdot r \right] | \leq C_n \, \sigma,
\end{equation}
\begin{equation}\label{eq:normaldersigma}
| \frac{d^n}{d \sigma^n} \left[ (\dot\x(t))^\bot \cdot r \right] | \leq C_n \, \sigma^{1-n}.
\end{equation}
\end{lemma}

The chain rule and Fa\`{a} di Bruno formula can then be invoked as previously, with the combined knowledge of (\ref{eq:normaldertau}), (\ref{eq:normaldersigma}), the growth of $H_1^{(1)}$ from Lemma \ref{teo:hankel1}, i.e.,
\[
| \frac{d^m}{d \phi^m} \left[ H_1^{(1)} (k \phi(s,t)) e^{-ik \phi(s,t)} \right] | \leq C_m \, \frac{1}{k \sigma} \, \sigma^{-m},
\]
as well as equations (\ref{eq:dphidtau}) and (\ref{eq:dphidsigma}) on the growth of the derivatives of $\phi$. It is straightforward to see that (\ref{eq:dadtau2}) is satisfied; for instance, the factor $\sigma$ from $(\dot\x(t))^\bot \cdot r$ and the leading $k$ in the expression of $G_1$ cancel out the $1/(k \sigma)$ in the formula for the derivatives of $H_1^{(1)}$. The rest of the argument involving the Fa\`{a} di Bruno formula is the same as previously. Equation (\ref{eq:dadsigma2}) follows from the same reasoning, and the observation that no $\sigma$ factor is gained upon differentiating $\phi$ in $\sigma.$

Since (\ref{eq:dadtau2}) and (\ref{eq:dadsigma2}) are at least as good as what they were in the case of $G_0$, the rest of the argument can proceed as previously with the same results; the ``off-strip'' and ``regular on-strip'' contributions, for instance, are unchanged from the $G_0$ scenario. The ``singular on-strip'' contribution however, corresponding to dyadic squares that intersect the diagonal, ought to be revisited since $G_1$ has a much milder singularity than $G_0$ near the diagonal.

The estimate of fast decay in $| m_1 + m_2 |$ and $\| \n \|$, namely (\ref{eq:onstrip-sing}), is a fortiori still valid. It appears, however, that the collective bound (\ref{eq:collective}) at scale $j$ can be improved to
\begin{equation}\label{eq:collective-DLP}
\sum_{\m, \n} | \< K \chi^{\mbox{diag}}_{j}  \rho(k\sigma), \vf_{j,\m,\n} \> |^2 \leq C \; 2^{-6j} k^2, 
\end{equation}
where $\chi^{\mbox{diag }}_{j}(s,t)$ refers to the $\sum_Q w_Q(s,t)$ over the squares $Q \in \mathcal{Q}_j$ at scale $j$ for which the support of $w_Q$ intersects the diagonal. The presence of an aggregation of windows $\chi^{\mbox{diag }}_{j}(s,t)$ is important here, as the study of coefficients corresponding to individual windows $w_Q$ would not give a sharp bound. Whether $\chi^{\mbox{diag }}_{j}(s,t)$ or $\rho(k\sigma)$ effectively determines the cutoff depends on the relative values of $j$ and $\log_2 k$.

Again, via a Plancherel argument, the scale-by-scale bound (\ref{eq:collective-DLP}) can be proved by passing to a system of Daubechies wavelets. We have
\begin{equation}\label{eq:waveletcoeff}
\< K \chi^{\mbox{diag}}_{j} \rho(k\sigma), \psi^\varepsilon_{j',\n} \>  = \int_{\mbox{supp }  \psi^\varepsilon_{j',\n}} \; \chi^{\mbox{diag }}_{j}(s,t) \rho(k\sigma) \; \frac{ik}{4} H_1^{(1)}(k \phi(s,t)) \; (\dot\x(t))^\bot \cdot r 
\end{equation}
\[
\qquad\qquad\qquad\qquad \times 2^{j'} \,  \psi^\varepsilon (2^{j'} s - n_1, 2^{j'} t - n_2) \, ds dt,
\]
where the scale of the wavelet relates to that of the window $w$ as $2j - j_0 \leq j' \leq 2j + j_0$. 

As previously we will use the vanishing moments of the wavelet to bring out a few $2^{-j'}$ factors. This time we will need up to three vanishing moments, i.e., we write the wavelet as
\[
\psi^\varepsilon(2^{j'} s - n_1, 2^{j'} t - n_2) = \left( 2^{-j'} \frac{d}{d s} \right)^M \tilde{\psi}^\varepsilon (2^{j'} s - n_1, 2^{j'} t - n_2), \qquad \varepsilon = \mbox{HL or HH},
\]
or as
\[
\psi^\varepsilon(2^{j'} s - n_1, 2^{j'} t - n_2) =  \left( 2^{-j'} \frac{d}{d t} \right)^M \tilde{\psi}^\varepsilon (2^{j'} s - n_1, 2^{j'} t - n_2), \qquad \varepsilon = \mbox{LH},
\]
where $M \leq 3$, and $ \tilde{\psi}^\varepsilon$ has the same support as $\psi^\varepsilon$. Before we let these derivatives act on the rest of the integrand, multiply and divide by $\phi(s,t) = \| \x(t) - \x(s) \|$ to get
\[
k H_1^{(1)}(k \phi(s,t)) \; (\dot\x(t))^\bot \cdot r = \left[ k\phi(s,t) H_1^{(1)}(k \phi(s,t)) \right] \; (\dot\x(t))^\bot \cdot \frac{\x(t) - \x(s)}{\| \x(t) - \x(s) \|^2}.
\]
We will then need the following lemma, which refines the results of equations (\ref{eq:dphidtau}), (\ref{eq:dphidsigma}), and Lemma \ref{teo:normalder}. It is proved in the Appendix.

\begin{lemma}\label{teo:normalder2}
Let $\phi(s,t) = \| \x(s) - \x(t) \|$. For every integer $m \leq 0$ there exists $C_m > 0$ such that, as long as $s \ne t$,
\begin{equation}\label{eq:dphids-refined}
| \left( \frac{d}{ds} \right)^m \phi(s,t) | \leq C_m,
\end{equation}
%\begin{equation}\label{eq:dphidt-refined}
%| \left( \frac{d}{dt} \right)^m \phi(s,t) | \leq C_m,
%\end{equation}
\begin{equation}\label{eq:dnrds-refined}
| \left( \frac{d}{ds} \right)^m  \left[ (\dot\x(t))^\bot \cdot \frac{\x(t) - \x(s)}{\| \x(t) - \x(s) \|^2} \right] | \leq C_m,
\end{equation}
%\begin{equation}\label{eq:dnrdt-refined}
%| \left( \frac{d}{dt} \right)^m  \left[ (\dot\x(t))^\bot \cdot \frac{\x(t) - \x(s)}{\| \x(t) - \x(s) \|^2} \right] | \leq C_m.
%\end{equation}
The same inequalities hold with $d/dt$ derivatives in place of $d/ds$ derivatives.
\end{lemma}

Let us first consider the wavelets whose support intersects the diagonal. There, the support of the wavelet is sufficiently small that $\chi^{\mbox{diag}}_j (s,t) = \rho(k \sigma) = 1$. One may integrate by parts only once in (\ref{eq:waveletcoeff}), because
\[
\nabla_{(s,t)} \phi(s,t) = \left( \dot\x(s) \cdot r, - \dot\x(t) \cdot r \right)
\]
is discontinuous, since the unit chord $r  = (\x(s)-\x(t))/\| \x(s) - \x(t) \|$ changes sign across the diagonal and $\dot\x(t) \ne 0$ there. The action of either $d/ds$ or $d/dt$ on the integrand after integration by parts gives:
\begin{itemize}
\item a $O(1)$ contribution for $\frac{d}{dx} (x H^{(1)}_1(x))$, by Lemma \ref{teo:hankel2};
\item by the chain rule, a $O(k)$ contribution for $\frac{d}{ds}(k\phi)$ and $\frac{d}{dt}(k \phi)$, because of Lemma \ref{teo:normalder2};
\item a $O(1)$ contribution for derivatives of $(\dot\x(t))^\bot \cdot r / \phi$, by Lemma \ref{teo:normalder2}.
\end{itemize}
The size of the support is $2^{-2j'}$, the wavelet comes with an $L^2$ normalization $2^{j'}$, one factor $2^{-j'}$ comes out of the vanishing moment, and $|j' - 2j| \leq $const.; hence diagonal wavelet coefficients obey the bound
\[
| \< K \chi^{\mbox{diag}}_{j} \rho(k\sigma), \psi^\varepsilon_{j',\n} \>  | \leq 2^{-4j} k. \qquad \mbox{(diagonal wavelets)}
\]
There are $O(2^{j'}) = O(2^{2j})$ such diagonal wavelets overall, hence the sum of squares of these coefficients is bounded by $2^{-6j} k^2$, in accordance with equation (\ref{eq:collective-DLP}).

Let us now treat the wavelets that do not intersect the diagonal $s=t$, and show that the same bound is valid. One will now need to integrate by parts three times in $s$ or $t$ to get three $2^{-j'}$ factors out, and gather the action of the derivatives on the rest of the integrand as follows.
\begin{itemize}
\item The factors $\rho(k \sigma) \chi^{\mbox{diag}}_j(s,t)$ are essentially multiplied by $\max(2^j, k)$ for each derivative.
\item By Lemma \ref{teo:hankel2}, the combination $x H^{(1)}_1(x)$ becomes $1/x$ when differentiated three times in $x$. This is $1/k\phi$ when $x = k \phi$.
\item Derivatives of $\phi$ in $s$ and $t$ remain $O(1)$ by Lemma \ref{teo:normalder2}, hence each derivative of $k \phi$ yields a factor $k$.
\item By Lemma \ref{teo:normalder2}, al derivatives of $(\dot\x(t))^\bot \cdot \frac{\x(t) - \x(s)}{\| \x(t) - \x(s) \|^2}$ remain bounded uniformly in $s$ and $t$.
\end{itemize}

The product rule yields many terms but the overall sum is controlled by the behavior of the ``extreme'' terms identified above, hence a factor $\max(2^{3j},k^3) + k^3 / k\phi$ under the integral sign. Since the wavelet has support well away from the diagonal, we can proceed as previously and bound $\phi^{-1}(s,t)$ by $q^{-1} 2^{j'}$ where $q$ is an integer indexing the distance between the diagonal and the center of the wavelet. Again, the support of the wavelet has area $O(2^{-2j'})$ and $j'$ is comparable to $2j$, hence we get a bound
\[
| \< K \chi^{\mbox{diag}}_{j} \rho(k\sigma), \psi^\varepsilon_{j',\n} \>  | \leq 2^{-8j} \left[ \max(2^{3j}, k^3) + k^2 q^{-1} 2^{2j} \right]. \qquad \mbox{(nondiagonal wavelets)}
\]

As seen previously, the number of wavelet coefficients is a $O(\min(2^{2j}/k, 2^{j}))$ across the diagonal (indexed by $q$), times $O(2^{2j})$ along the diagonal, for a total of $O(\min(2^{4j}/k, 2^{3j}))$. Hence we have
\begin{align*}
\sum_{j' \in [2j-j_0, 2j + j_0]} &\sum_{\varepsilon,\n} |\< K_Q \rho(k\sigma), \psi^\varepsilon_{j',\n} \>|^2 \\
&\lesssim \left[ \min\left(\frac{2^{4j}}{k}, 2^{3j}\right) \times 2^{-16j} \max(2^{6j}, k^6) \right] +  \left[ 2^{2j} \sum_q 2^{-16 j} k^4 q^{-2} 2^{4j} \right]\\
&\lesssim \left[ 2^{-12j} k^5 + 2^{-7j} \right] + \left[ 2^{-10j} k^4 \right] \\
&\lesssim \; 2^{-6j} k^2 \qquad\qquad \mbox{since } j \geq \frac{1}{2} \log_2 k + C.
\end{align*}
This is the desired decay rate, compatible with equation (\ref{eq:collective-DLP}).

We are now left with the task of verifying that (\ref{eq:collective-DLP}) implies the correct decay of the $\ell_p$ norm, as in equation (\ref{eq:Kmu-ellp3}). Let $p < 1$. Start by using H\"{o}lder's inequality (\ref{eq:Holder}) with $N = O(2^{j(2+\delta)})$---there are $O(2^j(1+\delta))$ wave atoms per square $Q$, and $O(2^{j})$ squares along the diagonal. We get
\[
\sum_{\m, \n} | \< K \chi^{\mbox{diag}}_{j}  \rho(k\sigma), \vf_{j,\m,\n} \> |^p \leq C \; \left[ ( 2^{-3j} k ) 2^{j(2+\delta)(1/p-1/2)} \right]^p = C \, k^p 2^{(2-4p + \delta' p)j},
\]
where $\delta' = \delta (\frac{1}{p} - \frac{1}{2})$. Finally, the $p$-triangle inequality asks to sum this bound over $j \geq \frac{1}{2} \log_2 k$. The sum is convergent provided $p > 1/2$ and $\delta'$ is taken sufficiently small. The result is
\[
\| K_\mu \|^p_{\ell_p(F)} \leq C_p \, k^p k^{\frac{1}{2}(2-4p+\delta' p)} = C_p \, k^{1-p + \delta''}, \qquad \forall \delta'' > 0
\]
After taking the $1/p$-th power, we fall exactly into scenario 3 for the $\ell_p$ summation, i.e., equation (\ref{eq:Kmu-ellp3}). The proof is complete.

\subsection{Proof of Corollary \ref{teo:cor}}

Passing to relative error estimates requires scaling $\eps$ by $1/\sqrt{k}$ and $\sqrt{k}$ respectively. Recall $k \geq 1$.

\begin{itemize}
\item If we invoke Theorem \ref{teo:main} for $G_0$, with an absolute error $\eps/\sqrt{k}$ in place of $\eps$, then the number of terms $|\Lambda_0|$ becomes $O((k \eps^{-2})^{1 + 1/\infty})$. Hence if we can show that $1/\sqrt{k} \leq C \| K^0 \|_{2}$, then (\ref{eq:G0rel}) follows.
\item If we invoke Theorem \ref{teo:main} for $G_1$, with an absolute error $\eps \sqrt{k}$ in place of $\eps$, then the number of terms $|\Lambda_1|$ becomes $O(k^{1+1/\infty} \eps^{-1/\infty} + (k \eps^{-2})^{1/3 + 1/\infty})$. Hence if we can show that $\sqrt{k} \leq C \| K^1 \|_2$, then (\ref{eq:G1rel}) follows.
\item The combination of the above two error bounds would show (\ref{eq:G0G1rel}), provided we can show that $\sqrt{k} \geq C \| K^1 - i \eta K^0 \|_2$ when $\eta \asymp k$.
\end{itemize}

Therefore, it suffices to establish the lower bounds on $\| K^0 \|_2, \| K^1 \|_2$, and $\| K^{(0,1)} \|_2$. By the tight frame property of wave atoms, the claim for $K^0$ is exactly $\int_{[0,1]^2} |G_0(s,t)|^2 ds dt \geq C/k.$ Set $\phi(s,t) = \| \x(s) - \x(t) \|$. As $k \phi(s,t) > c_0$, Lemma \ref{teo:hankel3} implies that $|G_0(s,t)| \leq C (k \phi)^{-1/2}$. If $k$ is sufficiently large, we can restrict the integration domain to the nonempty set $k \phi(s,t) > c_0$, and directly conclude. If $k$ is not large enough for this step, the integral is still a uniformly continuous and positive function of $k$, hence uniformly bounded away from zero.

The claim for $K^1$ is $\int_{[0,1]^2} |G_1(s,t)|^2 ds dt \geq k.$ The non-Hankel factors in the expression of $G_1$ play a minor role in evaluating this lower bound; we can consider them bounded away from zero on a large set $S_1$ to which the integral is restricted. Lemma \ref{teo:hankel3} then implies that for $k \phi(s,t) > c_1$ and $(s,t) \in S_1$, we have $|G_1(s,t)| \geq k (k \phi)^{-1/2}$. By the same reasoning as previously, this leads to the lower bound.

The claim for $K^{(0,1)}$ is $\int_{[0,1]^2} |G_1(s,t) - i \eta G_0(s,t)|^2 ds dt \geq k,$ with $\eta \asymp k$. The reasoning is here a little more complicated since $G_1$ and $k G_0$ are on the same order of magnitude. The presence of $- i$, however, prevents major cancellations---and is in fact chosen for that very reason. The asymptotic decay of $G_1 - i \eta G_0$ for large $\phi(s,t)$ can be studied from the integral formulation of the Hankel function used throughout the Appendix for proving the three Lemmas in section \ref{sec:hankel}. Without entering into details, we remark that the integral factor in (\ref{eq:Hn-integral}) is for $z$ large very near real-valued, with positive real part. The exponential factor $e^{-in\pi/2}$ shows that $H_1$ is then almost aligned with $-i H_0$. The particular combination $G_1 - i \eta G_0$ with $\eta > 0$ respects this quadrature property of Hankel functions, and produces no cancellation at all in the limit $z \to \infty$. So for $k \phi$ large enough, the claim follows; and if $k \phi$ is not large enough, we fall back on an argument of uniform continuity as previously.

\end{itemize}

%------------------------------------------------------------------------------
\section{Numerical experiments}\label{sec:num}

In this section, we provide several numerical examples to support the
sparsity results of the previous section. The three geometric objects
used in this section are displayed in Figure \ref{fig:geom}. For each
object, the boundary curve is represented using a small number of
Fourier coefficients. The last two examples have non-convex shapes
that typically result in multiple scattering effects. We report the
numerical results for the single layer kernel $G_0(s,t)$ in Section
\ref{sec:num_sl} and the results for the double layer kernel
$G_1(s,t)$ in Section \ref{sec:num_dl}. We omit the results of the
combined kernel $G_1(s,t) - i \eta G_0(s,t)$ as they are almost the
same as the single layer case.

%LD: maybe a comment to the effect that doing G_1-ietaG_0 would not show very much?

\begin{figure}[h!]
  \begin{center}
    \begin{tabular}{ccc}
      \includegraphics[height=1.8in]{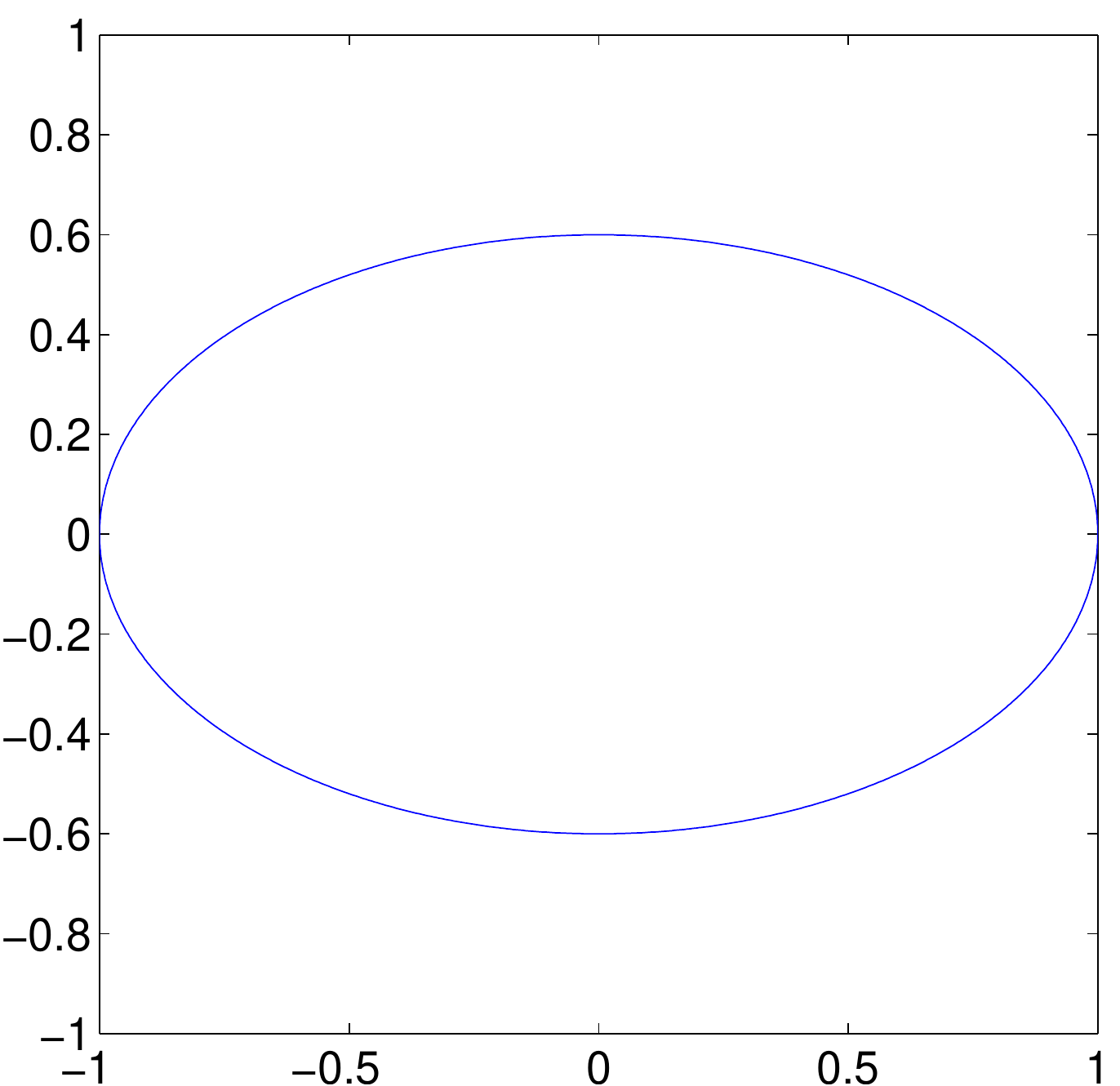}&
      \includegraphics[height=1.8in]{slkite0.pdf}&
      \includegraphics[height=1.8in]{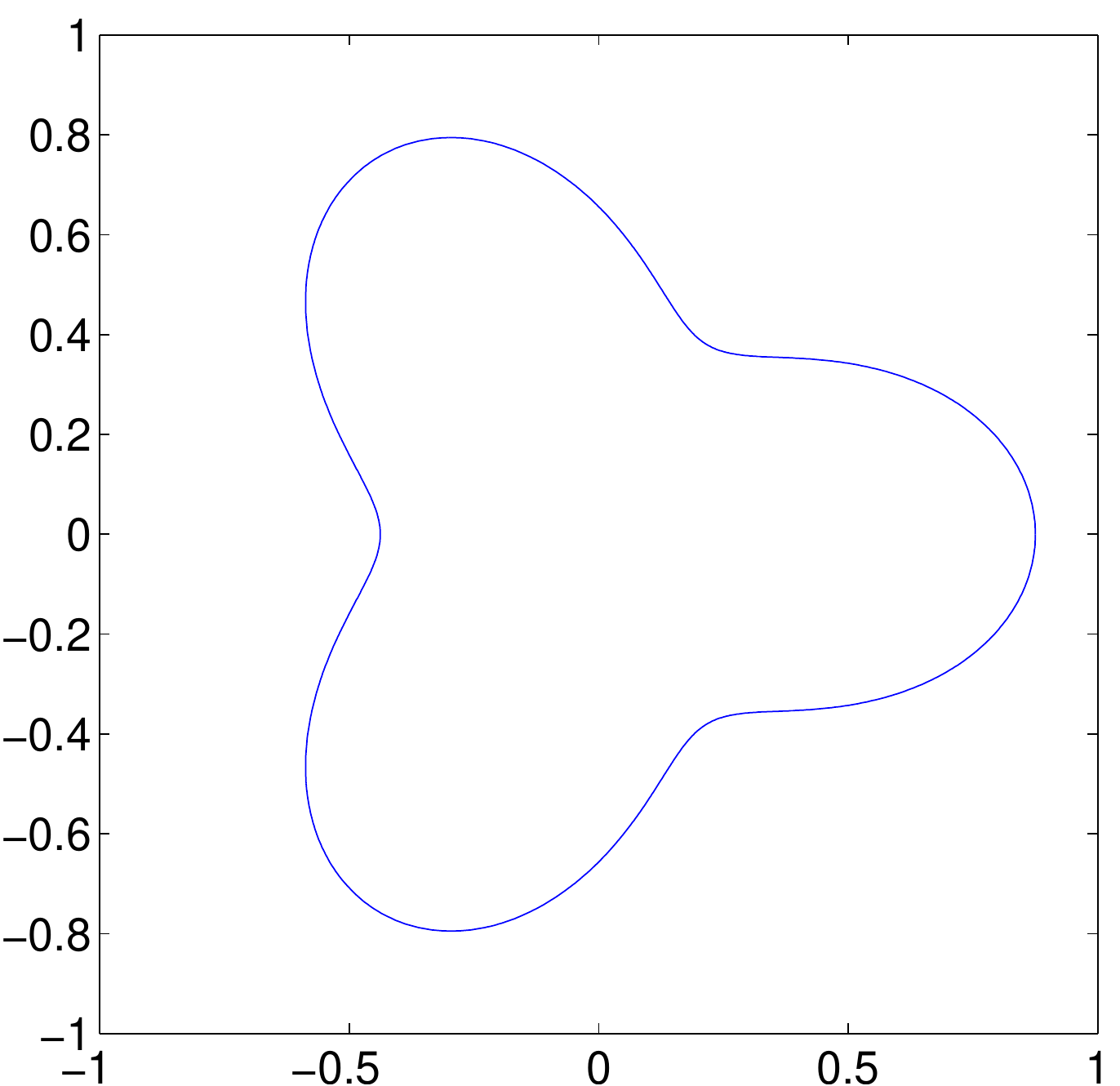}\\
      (a) & (b) & (c)
    \end{tabular}
  \end{center}
  \caption{The geometric objects used in the test examples.  (a): an
    ellipse. (b): a kite-shaped object. (c): a star-shaped object.}
  \label{fig:geom}
\end{figure}

%-----------------------------------------
\subsection{Single layer potential}\label{sec:num_sl}

We first study the single layer potential
\[
k \cdot G_0(s,t) = k \cdot \frac{i}{4} H^{(1)}_0(k \| \mathbf{x}(s) - \mathbf{x}(t) \|) \, \|\dot{\mathbf{x}}(t)\|.
\]
Notice that we use $k \cdot G_0(s,t)$ instead of $G_0(s,t)$, because the coupling constant $\eta$ in the integral equation \eqref{eq:bie} is of order $k$. Therefore, $k \cdot G_0(s,t)$ is more informative when we report the value at which coefficients are thresholded, and
the number of nonnegligible coefficients.

For each fixed $k$, we construct the discrete version of the operator
$k \cdot G_0(s,t)$ by sampling the boundary curve with $N = 8k$
quadrature points; this corresponds to about 8 points per wavelength
in these examples. Next, we scale the values at these quadrature
points with the high-order corrected trapezoidal quadrature rule from
\cite{K} in order to integrate the logarithmic singularity accurately.
This quadrature rule has the appealing feature of changing the weights
only locally close to the singularity. We then apply the two
dimensional wave atom transform to compute the coefficients $K^0_\mu
:= \< k G_0, \vf_\mu \>$. For a fixed accuracy $\eps$, we obtain the
sparsest approximant $\tilde{K}^0_\mu$ that satisfies
\[
\| K^0 - \tilde{K}^0 \|_{\ell_2(\mu)} \le \eps \| K^0 \|_{\ell_2(\mu)}
\]
by choosing the largest possible threshold value $\delta$ and setting
the coefficients less than $\delta$ in modulus to zero.  Equation
(\ref{eq:G0rel}) predicts that, as a function of $k$, the number of
wave atom coefficients defining $\tilde{K}^0$ should grow like
$k^{1+1/\infty}$.

\begin{table}[h!]
  \begin{center}
    \includegraphics[height=2in]{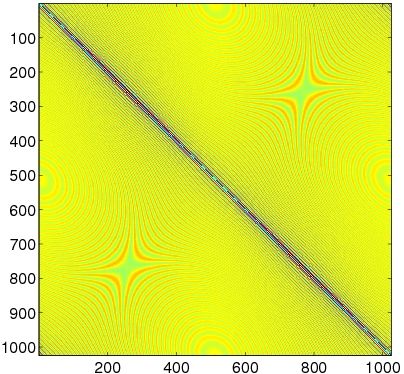}
    \includegraphics[height=2in]{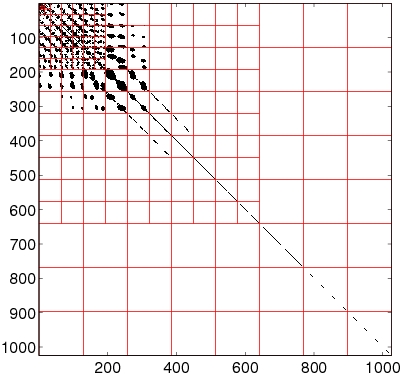}\\
    \vspace{0.1in}
    \begin{tabular}{|c|r|r|r|}
      \hline
      & $\eps=10^{-1}$ & $\eps=10^{-1.5}$ & $\eps=10^{-2}$ \\
      \hline
      $k=32  $ & 11 / 3.26e-2 / 9.53e-2  &    18 / 9.09e-3 / 2.88e-2   &   27 / 2.60e-3 / 7.93e-3      \\
      $k=64  $ &  9 / 3.37e-2 / 1.07e-1  &    16 / 9.20e-3 / 3.39e-2   &   28 / 2.39e-3 / 1.01e-2      \\
      $k=128 $ & 10 / 3.33e-2 / 9.80e-2  &    20 / 8.46e-3 / 3.29e-2   &   32 / 2.45e-3 / 1.04e-2      \\
      $k=256 $ &  9 / 3.63e-2 / 9.79e-2  &    18 / 9.33e-3 / 2.99e-2   &   30 / 2.46e-3 / 9.66e-3      \\
      $k=512 $ & 11 / 3.61e-2 / 9.79e-2  &    20 / 9.39e-3 / 3.07e-2   &   33 / 2.55e-3 / 9.86e-3      \\
      $k=1024$ &  9 / 4.00e-2 / 9.78e-2  &    17 / 1.01e-2 / 3.12e-2   &   29 / 2.72e-3 / 9.81e-3      \\
      \hline
    \end{tabular}
  \end{center}
  \caption{Single layer potential for the ellipse.
    Top left: the real part of the operator for $k=128$.
    Top right: the sparsity pattern of the operator under the wave atom basis 
    for $k=128$ and $\eps=10^{-2}$. Each black pixel stands for a nonnegligible coefficient.
    Bottom: For different combinations of $k$ and $\eps$, the number of nonnegligible entries
    per row $|\Delta_0|/N$, the threshold value $\delta$, and the estimated $L^2$ operator norm $\eps_{L^2}$ (from left to right).
  }
  \label{tbl:slellp}
\end{table}

\begin{table}[h!]
  \begin{center}
    \includegraphics[height=2in]{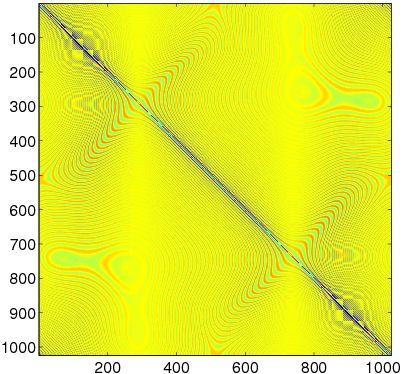}
    \includegraphics[height=2in]{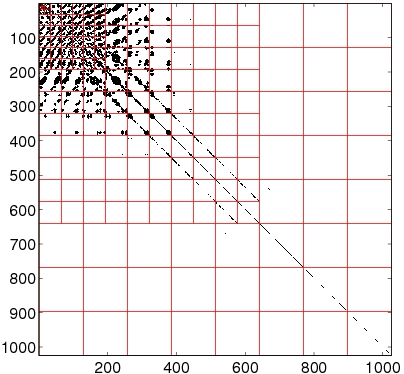}\\
    \vspace{0.1in}
    \begin{tabular}{|c|r|r|r|}
      \hline
      & $\eps=10^{-1}$ & $\eps=10^{-1.5}$ & $\eps=10^{-2}$ \\
      \hline
      $k=32  $ & 14 / 2.92e-2 / 9.50e-2   &   25 / 8.35e-3 / 3.07e-2  &    37 / 2.69e-3 / 8.79e-3\\      
      $k=64  $ & 15 / 2.70e-2 / 9.77e-2   &   30 / 7.66e-3 / 3.45e-2  &    46 / 2.41e-3 / 9.98e-3  \\    
      $k=128 $ & 17 / 2.68e-2 / 1.04e-1   &   34 / 7.49e-3 / 3.35e-2  &    53 / 2.26e-3 / 1.04e-2    \\  
      $k=256 $ & 17 / 2.70e-2 / 1.06e-1   &   35 / 7.27e-3 / 3.36e-2  &    58 / 2.07e-3 / 1.05e-2      \\
      $k=512 $ & 17 / 2.85e-2 / 1.11e-1   &   35 / 7.58e-3 / 3.43e-2  &    58 / 2.09e-3 / 1.07e-2      \\
      $k=1024$ & 17 / 2.89e-2 / 1.03e-1   &   35 / 7.75e-3 / 3.33e-2  &    60 / 2.07e-3 / 1.07e-2\\
      \hline
    \end{tabular}
  \end{center}
  \caption{Single layer potential for the kite-shaped object.
    Top left: the real part of the operator for $k=128$.
    Top right: the sparsity pattern of the operator under the wave atom basis 
    for $k=128$ and $\eps=10^{-2}$.
    Bottom: For different combinations of $k$ and $\eps$, $|\Delta_0|/N$,  $\delta$, and $\eps_{L^2}$.
  }
  \label{tbl:slkite}
\end{table}

\begin{table}[h!]
  \begin{center}
    \includegraphics[height=2in]{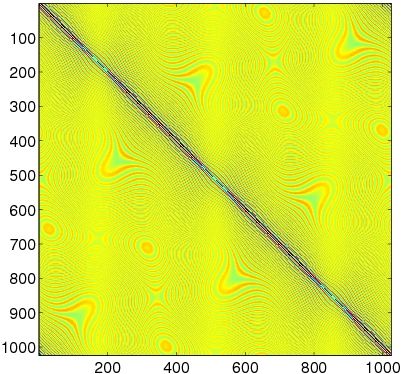}
    \includegraphics[height=2in]{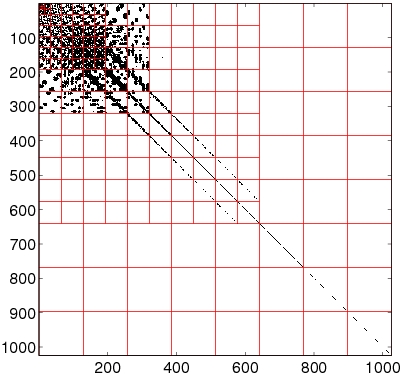}\\
    %\vspace{0.1in}
    \begin{tabular}{|c|r|r|r|}
      \hline
      & $\eps=10^{-1}$ & $\eps=10^{-1.5}$ & $\eps=10^{-2}$ \\
      \hline
      $k=32  $ & 14 / 2.76e-2 / 1.00e-1 &     23 / 8.74e-3 / 3.14e-2 &     32 / 2.50e-3 / 9.72e-3      \\
      $k=64  $ & 14 / 2.55e-2 / 8.97e-2 &     25 / 7.68e-3 / 2.83e-2 &     38 / 2.39e-3 / 8.28e-3      \\
      $k=128 $ & 18 / 2.39e-2 / 8.36e-2 &     33 / 7.01e-3 / 2.64e-2 &     49 / 2.19e-3 / 8.06e-3      \\
      $k=256 $ & 18 / 2.32e-2 / 9.48e-2 &     35 / 6.66e-3 / 2.90e-2 &     56 / 1.98e-3 / 9.33e-3      \\
      $k=512 $ & 19 / 2.34e-2 / 9.30e-2 &     40 / 6.19e-3 / 2.94e-2 &     66 / 1.84e-3 / 9.30e-3      \\
      $k=1024$ & 18 / 2.42e-2 / 9.50e-2 &     38 / 6.41e-3 / 3.01e-2 &     66 / 1.77e-3 / 9.41e-3       \\
      \hline
    \end{tabular}
  \end{center}
  \caption{Single layer potential for the star-shaped object.
    Top left: the real part of the operator for $k=128$.
    Top right: the sparsity pattern of the operator under the wave atom basis 
    for $k=128$ and $\eps=10^{-2}$.
    Bottom: For different combinations of $k$ and $\eps$, $|\Delta_0|/N$,  $\delta$, and $\eps_{L^2}$.
  }
  \label{tbl:slstar}
\end{table}

For each example in Figure \ref{fig:geom}, we perform the test for
different combinations of $(k,\eps)$ with $k=32,64,\ldots,1024$ and
$\eps=10^{-1},10^{-1.5}$, and $10^{-2}$. The numerical results for the
three examples are summarized in Tables \ref{tbl:slellp},
\ref{tbl:slkite}, and \ref{tbl:slstar}, respectively. In each table,
\begin{itemize}
\item The top left plot is the real part of the single layer potential
  in the case of $k=128$. This plot displays coherent oscillatory
  patterns for which the wave atom frame is well suited.
\item The top right plot is the sparsity pattern of the operator
  under the wave atom basis for $k=128$ and $\eps=10^{-2}$. Each black
  pixel stands for a nonnegligible coefficient. The coefficients are
  organized in a way similar to the usual ordering of 2D wave atom
  coefficients: each block contains the wave atom coefficients of a
  fixed frequency index $(j,\m)$, and the blocks are ordered such that
  the lowest frequency is located at the top left corner while the
  highest frequency at the bottom right corner. Within a block, the
  wave atom coefficients of frequency index $(j,\m)$ are ordered
  according to their spatial locations. The multiscale nature of the
  wave atom frame can be clearly seen from this plot.
\item The table at the bottom gives, for different combinations of $k$
  and $\eps$, the number of nonnegligible coefficients per row
  $|\Delta_0|/N$ , the threshold value $\delta$ (coefficients below
  this value in modulus are put to zero) and the $L^2$-to-$L^2$ norm
  operator error $\eps_{L^2}$ estimated using random test functions.
\end{itemize}

In these tables, the number of significant coefficients per row
$\Delta_0/N$ grows very slowly as $k$ doubles and reaches a constant
level for large values of $k$. This match well with the theoretical
analysis in Section \ref{sec:sparsity}. The threshold value $\delta$
remains roughly at a constant level as $k$ grows, which is quite
different from the results obtained using wavelet packet bases
\cite{DL1,DL2,G,H} where the threshold value in general decreases as
$k$ grows. The estimated $L^2$-to-$L^2$ operator error $\eps_{L^2}$ is
very close to the prescribed accuracy $\eps$ in all cases. This
indicates that, in order to get an approximation within accuracy
$\eps$ in operator norm, one can simply truncate the nonstandard form
of the operator in the wave atom frame with the same accuracy.

%LD: we could explain why epsilon was not chosen smaller, but it is not necessary.

%-------------------
\subsection{Double layer potential} \label{sec:num_dl}

We now consider the double layer potential
\[
G_1(s,t) = \frac{i k}{4} H_1^{(1)}(k \| \mathbf{x}(s)-\mathbf{x}(t) \|) \, \frac{\mathbf{x}(s) -\mathbf{x}(t)}{\| \mathbf{x}(s) - \mathbf{x}(t) \|} \cdot n_{\mathbf{x}(t)} \, \|\dot{\mathbf{x}}(t)\|.
\]

For each fixed $k$, the discrete version of $G_1(s,t)$ is constructed
by sampling at $N=8 k$ points and using trapezoidal quadrature rule.
The coefficients $K^1_\mu := \< G_1, \vf_\mu \>$ are calculated using
the two dimensional wave atom transform and the approximant
$\tilde{K}^1_\mu$ is constructed in the same way as the single layer
potential case.

The results of the double layer potentials for the three examples are
summarized in Tables \ref{tbl:dlellp}, \ref{tbl:dlkite}, and
\ref{tbl:dlstar}, respectively. These results are qualitatively
similar to the ones of the singular layer potential. However, the
coefficients of the double layer potential exhibit better sparsity
pattern for the simple reason that the double layer potential operator
has a singularity much weaker than logarithmic along the diagonal (where
$s=t$) for objects with smooth boundary. Therefore, for a fixed
accuracy $\eps$, the number of wave atoms required along the diagonal
for the double layer potential is smaller than the number for the
singular layer potential. This is clearly shown in the sparsity
pattern plots in Tables \ref{tbl:dlellp}, \ref{tbl:dlkite}, and
\ref{tbl:dlstar}.

\begin{table}[h!]
  \begin{center}
    \includegraphics[height=2in]{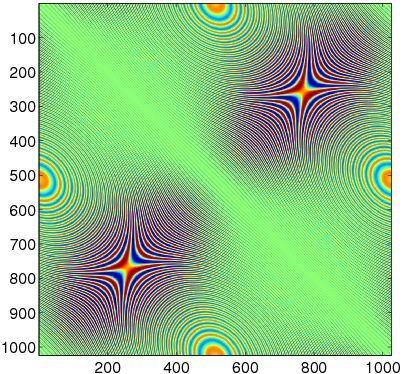}
    \includegraphics[height=2in]{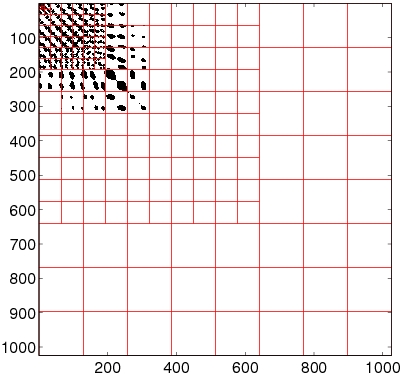}\\
    \vspace{0.1in}
    \begin{tabular}{|c|r|r|r|}
      \hline
      & $\eps=10^{-1}$ & $\eps=10^{-1.5}$ & $\eps=10^{-2}$ \\
      \hline
      $k=32  $ & 10 / 1.39e-2 / 1.29e-1 &     16 / 4.61e-3 / 3.65e-2 &     21 / 1.57e-3 / 1.19e-2       \\
      $k=64  $ & 9  / 1.40e-2 / 1.00e-1 &     15 / 4.41e-3 / 3.23e-2 &     21 / 1.34e-3 / 1.01e-2      \\
      $k=128 $ & 11 / 1.21e-2 / 9.50e-2 &     20 / 3.74e-3 / 3.00e-2 &     28 / 1.12e-3 / 9.94e-3      \\
      $k=256 $ & 10 / 1.28e-2 / 9.80e-2 &     18 / 3.70e-3 / 3.28e-2 &     28 / 1.07e-3 / 1.02e-2      \\
      $k=512 $ & 13 / 1.16e-2 / 1.00e-1 &     22 / 3.42e-3 / 3.24e-2 &     33 / 9.77e-4 / 1.00e-2      \\
      $k=1024$ & 11 / 1.21e-2 / 9.98e-2 &     20 / 3.52e-3 / 3.29e-2 &     30 / 9.83e-4  / 1.01e-2   \\
      \hline
    \end{tabular}
  \end{center}
  \caption{Double layer potential for the ellipse.
    Top left: the real part of the operator for $k=128$.
    Top right: the sparsity pattern of the operator under the wave atom basis 
    for $k=128$ and $\eps=10^{-2}$.
    Bottom: For different combinations of $k$ and $\eps$, $|\Delta_1|/N$,  $\delta$, and $\eps_{L^2}$.
  }
  \label{tbl:dlellp}
\end{table}

\begin{table}[h!]
  \begin{center}
    \includegraphics[height=2in]{dlkite1.jpg}
    \includegraphics[height=2in]{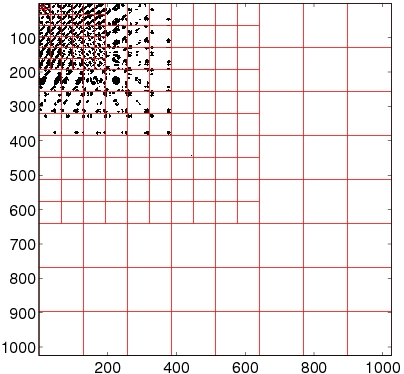}\\
    \vspace{0.1in}
    \begin{tabular}{|c|r|r|r|}
      \hline
      & $\eps=10^{-1}$ & $\eps=10^{-1.5}$ & $\eps=10^{-2}$ \\
      \hline
      $k=32  $& 12 / 1.64e-2 / 9.16e-2 &    19 / 5.16e-3 / 2.82e-2 &     27 / 1.64e-3 / 8.71e-3      \\
      $k=64  $& 13 / 1.40e-2 / 1.02e-1 &    23 / 4.37e-3 / 3.18e-2 &     34 / 1.31e-3 / 1.10e-2      \\
      $k=128 $& 16 / 1.27e-2 / 9.11e-2 &    29 / 3.82e-3 / 2.80e-2 &     43 / 1.18e-3 / 9.27e-3      \\
      $k=256 $& 16 / 1.22e-2 / 9.56e-2 &    31 / 3.46e-3 / 3.23e-2 &     49 / 1.02e-3 / 1.01e-2      \\
      $k=512 $& 18 / 1.20e-2 / 9.98e-2 &    33 / 3.42e-3 / 3.09e-2 &     52 / 9.71e-4 / 9.81e-3      \\
      $k=1024$& 18 / 1.17e-2  / 9.57e-2 &    35 / 3.26e-3 / 3.02e-2 &      55 / 9.07e-4 / 1.03e-2 \\
      \hline
    \end{tabular}
  \end{center}
  \caption{Double layer potential for the kite-shaped object.
    Top left: the real part of the operator for $k=128$.
    Top right: the sparsity pattern of the operator under the wave atom basis 
    for $k=128$ and $\eps=10^{-2}$.
    Bottom: For different combinations of $k$ and $\eps$, $|\Delta_1|/N$,  $\delta$, and $\eps_{L^2}$.
  }
  \label{tbl:dlkite}
\end{table}

\begin{table}[h!]
  \begin{center}
    \includegraphics[height=2in]{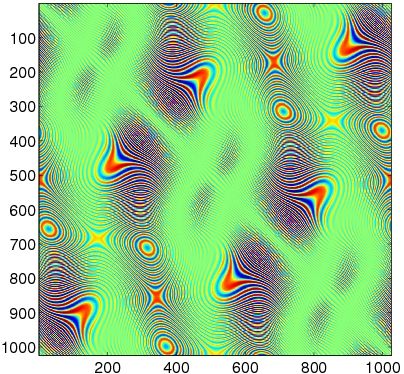}
    \includegraphics[height=2in]{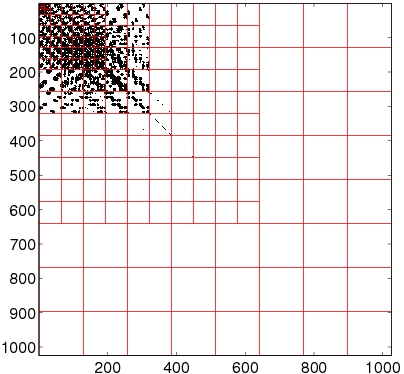}\\
    \vspace{0.1in}
    \begin{tabular}{|c|r|r|r|}
      \hline
      & $\eps=10^{-1}$ & $\eps=10^{-1.5}$ & $\eps=10^{-2}$ \\
      \hline
      $k=32  $ & 13 / 1.43e-2 / 1.20e-1 &     20 / 5.15e-3 / 3.31e-2  &    25 / 1.62e-3 / 1.15e-2 \\     
      $k=64  $ & 14 / 1.25e-2 / 9.20e-2 &     23 / 4.23e-3 / 2.80e-2  &    31 / 1.43e-3 / 9.03e-3   \\   
      $k=128 $ & 19 / 1.08e-2 / 1.08e-1 &     32 / 3.42e-3 / 3.27e-2  &    44 / 1.19e-3 / 1.10e-2     \\ 
      $k=256 $ & 20 / 1.00e-2 / 1.02e-1 &     36 / 3.08e-3 / 3.13e-2  &    52 / 9.77e-4 / 1.07e-2      \\
      $k=512 $ & 23 / 8.99e-3 / 1.01e-1 &     44 / 2.66e-3 / 3.24e-2  &    66 / 8.22e-4 / 1.01e-2      \\
      $k=1024$ & 22 / 9.06e-3 / 1.04e-1 &     43 / 2.54e-3 / 3.21e-2  &    69 / 7.46e-4 / 9.46e-3      \\
      \hline
    \end{tabular}
  \end{center}
  \caption{Double layer potential for the star-shaped object.
    Top left: the real part of the operator for $k=128$.
    Top right: the sparsity pattern of the operator under the wave atom basis 
    for $k=128$ and $\eps=10^{-2}$.
    Bottom: For different combinations of $k$ and $\eps$, $|\Delta_1|/N$,  $\delta$, and $\eps_{L^2}$.
  }
  \label{tbl:dlstar}
\end{table}

%LD: by the way, I've been very careful with the word "optimal", I hope this is OK.

%----------------------------------------------------------------------------
%\section{Discussion}

%Mention where this fits in the literature.

%[mirror-extended wave atoms and open-ended scatterers?]

%Comment on $ell_p$ spaces and $C^k$ functions -- why we assumed $C^\infty$.

%Proof techniques and how this relates to work on curvelets.

%Do not exclude that the $\epsilon^{-1 + \delta}$ can be improved upon using log factors.

\appendix

\section{Additional proofs}

{\bf Proof of Lemma \ref{teo:hankel1}.}

Following Watson's treatise \cite{Wat}, the Hankel function can be expressed by complex contour integration as
\begin{equation}\label{eq:Hn-integral}
H_n^{(1)}(z) = \left( \frac{2}{\pi z} \right)^{1/2} \, \frac{\exp i (z-\frac{n \pi}{2} - \frac{\pi}{4})}{\Gamma(n-\frac{1}{2})} \, \int_0^{\infty e^{i \beta}} e^{-u} u^{n-1/2} \left( 1 + \frac{iu}{2 z} \right)^{n-1/2} \, du,
\end{equation}
where $-\pi/2 < \beta < \pi/2$. For us, $z$ is real and positive, and we take $\beta = 0$ for simplicity.

Let us first treat the case $m = 0$ (no differentiations) and $n > 0$. We can use the simple bound
\[
| 1 + \frac{iu}{2z} |^{n-\frac{1}{2}} \leq C_n \, \left( 1 + |\frac{u}{z}|^{n-\frac{1}{2}} \right)
\]
to see that the integral, in absolute value, is majorized by $C_n \, (1 + z^{-n+1/2})$. Hence the Hankel function itself is bounded by $C_n \, (z^{-1/2} + z^{-n})$. This establishes the first two expressions in (\ref{eq:nonosc}) in the case $m = 0$. 

The case $m = n=0$ is treated a little differently because the integrand in (\ref{eq:Hn-integral}) develops a $1/u$ singularity near the origin as $z \to 0$. We have
\begin{equation}\label{eq:maj}
| 1 + \frac{iu}{2z} |^{-\frac{1}{2}} = \left( 1 +  \left( \frac{u}{2z} \right)^2 \right)^{-1/4} \leq C \, \min \left(1, \left( \frac{u}{z} \right)^{-1/2} \right),
\end{equation}
hence the integral in (\ref{eq:Hn-integral}) is bounded in modulus by a constant times
\[
\int_0^z e^{-u} u^{-1/2} \, du + z^{1/2} \int_z^\infty e^{-u} u^{-1} \, du \leq C \, ( z^{1/2} + z^{1/2} |\log z|).
\]
With the $z^{-1/2}$ factor from (\ref{eq:Hn-integral}), the resulting bound is $C\, (1 + |\log z|)$ as desired (third equation). When $z > 1$, we can improve this to $C (1 + z^{1/2} e^{-z}) \leq C$, which gives the first equation when $m = n = 0$.

For the case $m > 0$, it suffices to apply Leibniz's rule inductively and observe that each derivative produces a factor $x^{-1}$ without changing the power of $k$. In particular,
\begin{itemize}
\item With $\alpha \ne 0$,
\[
\frac{d}{dx} \left[ \left(kx \right)^{-\alpha} \right] = -\alpha \left( kx \right)^{-\alpha} \frac{1}{x},
\]
hence the power of $k$ is preserved and one negative power of $x$ is created.
\item Derivatives acting on isolated negative powers of $x$ also produce a $x^{-1}$ factor without affecting the dependence on $k$.
\item As for the $x$-dependence under the integral sign, with $\alpha \ne 0$, we arrange the factors as
\[
\frac{d}{dx} \left( \left(1 + \frac{iu}{2kx} \right)^{-\alpha} \right) = \alpha \left[ \frac{i u}{2 kx}  \, \left(1 + \frac{iu}{2kx} \right)^{-1} \right] \, \left(1 + \frac{iu}{2kx} \right)^{-\alpha} \, \frac{1}{x}.
\]
The factor in square brackets is bounded by 1 in modulus, hence the dependence on $k$ is unchanged. The factor $1/x$ is the  only modification in the dependence on $x$. Subsequent differentiations will only act on factors that we have already treated above: powers of $kx$, powers of $x$, and powers of $1 + \frac{iu}{2kx}$. This finishes the proof.

\end{itemize}

\bigskip

{\bf Proof of Lemma \ref{teo:hankel2}.}

As previously, we use the integral formulation to get
\[
x H^{(1)}_1(x) = f(x) \, \int_0^{\infty} e^{-u} u^{1/2} \left( x + i \frac{u}{2} \right)^{1/2} \, du,
\]
where $f(x)$ is the exponential factor, and already obeys $|f^{(n)}(x)| \leq C$ for all $n \geq 0$. We denote the integral factor by $I(x)$; its derivatives are
\[
I^{(n)}(x) = C_n \int_0^{\infty} e^{-u} u^{1/2} \left( x + i \frac{u}{2} \right)^{\frac{1}{2} - n} \, du,
\]
where $C_n$ is a numerical constant. In a manner analogous to the proof of Lemma \ref{teo:hankel1}, we can bound
\[
| x + i \frac{u}{2} |^{\frac{1}{2}-n} = \left( x^2 + \frac{u^2}{4} \right)^{\frac{1}{4} - \frac{n}{2}} \leq C_n \, \left( \max(x,u) \right)^{\frac{1}{2} - n}.
\]
It follows that
\[
|I^{(n)}(x)| \leq C_n \left[ \int_0^x e^{-u} u^{1/2} x^{1/2-n} \, du + \int_x^\infty e^{-u} u^{1-n} \, du \right].
\]
In the first term we can use $e^{-u} \leq 1$ and bound the integral by a constant times $x^{2-n}$. The integrand of the second term has a singularity near $u = 0$ that becomes more severe as $n$ increases; this term is bounded by $O(1)$ if $n=0$ or $1$, by $O(1+|\log(x)|)$ if $n = 2$, and by $O(x^{2-n})$ if $n > 2$.

\bigskip

{\bf Proof of Lemma \ref{teo:hankel3}.}

Consider equation (\ref{eq:Hn-integral}) again, and take $z$ real. For large values of $z$, the factor $(1+iu/2z)^{n-1/2}$ is close to $1$; more precisely, it is easy to show that for each $n \geq 0$, there exists $c_n > 0, d_n > 0$ for which
\[
| \left( 1 + \frac{iu}{2z} \right)^{n-1/2} - 1 | \leq d_n \frac{u}{z}, \qquad \mbox{if} \qquad u \leq c_n z.
\]

%More precisely let $\gamma = iu/2z$ and compute
%\[
%\left( 1 + \gamma \right)^{n-1/2} - 1 = \frac{(1+ \gamma)^{2n-1} - 1}{(1+\gamma)^{n-1/2} + 1} = \frac{- \gamma}{(1+\gamma)^{1/2} + 1 + \gamma}.
%\]
%The denominator is in modulus always greater than 1/2. as soon as $|\gamma| \leq 1/2$. As a result,
%\[
%| \left( 1 + \frac{iu}{2z} \right)^{-1/2} - 1 | \leq \frac{u}{z}, \qquad \mbox{when} \qquad u \leq z,
%\]
%where $u = |u|$ as previously. 

We can insert this estimate in (\ref{eq:Hn-integral}) and split the integral into two parts to obtain
\begin{align*}
| H_n^{(1)}(z) | \, &\Gamma(n - 1/2) \left( \frac{\pi z}{2} \right)^{1/2} \, - \, | \int_0^{c_n z} e^{-u} u^{-1/2} du | \geq \\
&- d_n \int_0^{z c_n} e^{-u} u^{-1/2} \frac{u}{z} du - C \int_{z c_n}^{\infty} e^{-u} u^{-1/2} du.
\end{align*}
The constant $C$ in the last term comes from equation (\ref{eq:maj}). The first term in the right-hand side is a $O(z^{-1})$, and the second term is a $O(e^{-z})$. At the expense of possibly choosing increasing the value of $c_n$, the second term in the left-hand side can manifestly be made to dominate the contribution of the right hand side, proving the lemma.

\bigskip

{\bf Proof of Lemma \ref{teo:normalder}.}

We start by writing
\[
(\dot\x(t))^\bot \cdot r = (\dot\x(t))^\bot \cdot \frac{\x(t) + \sigma \dot\x(t) - \x(s)}{\| \x(t) - \x(s) \|}.
\]
By Taylor's theorem, $|\x(t) + \sigma \dot\x(t) - \x(s)| \leq C \sigma^2$, hence $|(\dot\x(t))^\bot \cdot r| \leq C \sigma$.  Derivatives are then treated by induction; recall that $\frac{d}{d\tau} = \frac{d}{ds} + \frac{d}{dt}$ and $\frac{d}{d\sigma} = \frac{d}{ds} - \frac{d}{dt}$;
\begin{itemize}
\item any number of $\tau$ or $\sigma$ derivatives acting on $(\dot\x(t))^\bot$ leave it a $O(1)$;
\item $\tau$ derivatives acting on $\x(t) + \sigma \dot\x(t) - \x(s)$ leave it a $O(\sigma^2)$ while each $\sigma$ derivative removes an order of $\sigma$;
\item $\tau$ derivatives acting on $\| \x(t) - \x(s) \|^{-m}$ leave it a $O(\sigma^{-m})$ ($m$ is generic) while each $\sigma$ derivative removes an order of $\sigma$.
\end{itemize}
This shows (\ref{eq:normaldertau}) and (\ref{eq:normaldersigma}).

\bigskip

{\bf Proof of Lemma \ref{teo:normalder2}.}

Some cancellations will need to be quantified in this proof, that were not a concern in the justification of previous coarser estimates like Lemma \ref{teo:normalder}.

Without loss of generality, assume that $\x(t) = (0,0)$, $n_{\x(t)} = (0,1)$, and that we have performed a change of variables such that the curve is parametrized as the graph $\x(s) = (s,f(s))$ of some function $f \in C^\infty$ obeying $|f(s)| \leq C s^2$. This latter change of variables would contribute a bounded multiplicative factor that would not compromise the overall estimate.

By symmetry, if (\ref{eq:dphids-refined}) is true for $d/ds$ derivatives, then it will be true for $d/dt$ derivatives as well. Without loss of generality let $s > 0$. Then we have
\[
\phi(s,t) = \| \x(s) - \x(t) \| = \sqrt{s^2 + f^2(s)} = s \sqrt{1 + \frac{f^2(s)}{s^2}}.
\]
Since $|f(s)| \leq C s^2$ and $C^\infty$, the ratio $f^2(s)/s^2$ is also bounded for $s \lesssim 1$, and of class $C^\infty$. Being a composition of $C^\infty$ functions, the whole factor $\sqrt{1+f^2/s^2}$ is therefore also of class $C^\infty$, which proves (\ref{eq:dphids-refined}).

As for (\ref{eq:dnrds-refined}) with $d/ds$ derivatives, we can write $(\dot\x(t))^\bot = (0,1)$ and
\[
(\dot\x(t))^\bot \cdot \frac{\x(t) - \x(s)}{\| \x(t) - \x(s) \|^2} = \frac{- f(s)}{s^2 + f^2(s)} = - \left( \frac{f(s)}{s^2} \right) \frac{1}{1 + \frac{f^2(s)}{s^2}}.
\]
(The $1/\| \dot\x(t) \|$ does not pose a problem since it is $C^\infty$.) Again, since $|f(s)| \leq C s^2$, both ratios $f(s)/s^2$ and $f^2(s)/s^2$ are themselves bounded and of class $C^\infty$. The factor $\frac{1}{1 + \frac{f^2(s)}{s^2}}$ is the composition of two $C^\infty$ functions, hence also of class $C^\infty$.

The symmetry argument is not entirely straightforward for justifying (\ref{eq:dnrds-refined}) with $d/dt$ derivatives. Symmetry $s \leftrightarrow t$ only allows to conclude that
\[
| \left( \frac{d}{dt} \right)^m  \left[ (\dot\x(s))^\bot \cdot \frac{\x(t) - \x(s)}{\| \x(t) - \x(s) \|^2} \right] | \leq C_m,
\]
where $(\dot\x(s))^\bot$ appears in place of the desired $(\dot\x(t))^\bot$. Hence it suffices to show that the $d/dt$, or equivalently the $d/ds$ derivatives of
\[
\left[ (\dot\x(t))^\bot - (\dot\x(s))^\bot \right] \cdot \frac{\x(t) - \x(s)}{\| \x(t) - \x(s) \|^2}
\]
stay bounded. Using our frame in which the curve is a graph, we find
\[
(\dot\x(t))^\bot - (\dot\x(s))^\bot = (0,1) - (-f'(s),1) = (f'(s),0)
\] 
As a result,
\[
\left[ (\dot\x(t))^\bot - (\dot\x(s))^\bot \right] \cdot \frac{\x(t) - \x(s)}{\| \x(t) - \x(s) \|^2} = \frac{-s f'(s)}{s^2 + f^2(s)} = - \left( \frac{f'(s)}{s} \right)  \frac{1}{1 + \frac{f^2(s)}{s^2}}.
\]
Since $|f'(s)| \leq s$, we are again in presence of a combination of $C^\infty$ functions that stays infinitely differentiable.


\begin{thebibliography}{99}

\bibitem{AS} M. Abramowitz and I. A. Stegun. \emph{Handbook of
    Mathematical Functions with Formulas, Graphs, and Mathematical
    Tables.} New York: Dover, 1972.

\bibitem{AM} J. P. Antoine, R. Murenzi, Two-dimensional directional
  wavelets and the scale-angle representation. \emph{Sig. Process.}
  \textbf{52} (1996), 259--281.
  
\bibitem{ABCIS} A. Averbuch, E. Braverman, R. Coifman, M. Israeli and
  A. Sidi, Efficient computation of oscillatory integrals via adaptive
  multiscale local Fourier bases, \emph{App. Comput. Harmon. Anal.},
  \textbf{9(1)} (2000) 19--53

\bibitem{BCR} G. Beylkin, R. Coifman, and V. Rokhlin, Fast wavelet
  transforms and numerical algorithms. I.  \emph{Comm. Pure Appl.
    Math.} 44 (1991), no. 2, 141--183.

\bibitem{BCG} B. Bradie, R. Coifman, and A. Grossmann, Fast numerical
  computations of oscillatory integrals related to acoustic
  scattering, \emph{App. Comput. Harmon. Anal.}, \textbf{1} (1993)
  94--99.
  
\bibitem{CD} E.~J. Cand\`{e}s and D.~L. Donoho. New tight
  frames of curvelets and optimal representations of objects with
  piecewise-$C^2$ singularities. \emph{Comm. on Pure and Appl.  Math.}
  \textbf{57} (2004), 219--266.
  
\bibitem{Ch} H. Cheng, W. Y. Crutchfield, Z. Gimbutas, L F. Greengard,
  J. F. Ethridge, J. Huang, V. Rokhlin, N. Yarvin, and J. Zhao, A
  wideband fast multipole method for the Helmholtz equation in three
  dimensions, \emph{J. Comput. Phys., 216} (2006), pp. 300--325.
  
\bibitem{CF} A. C\'{o}rdoba, C. Fefferman, Wave packets and Fourier
  integral operators. \emph{Comm. PDE} \textbf{3(11)} (1978),
  979--1005.

\bibitem{ThesisDemanet} L. Demanet, \emph{Curvelets, Wave Atoms and
    Wave Equations}, Ph. D. thesis, California Institute of
  Technology, 2006.
  
\bibitem{mirror-extended} L. Demanet and L. Ying, Curvelets and wave
  atoms for mirror-extended images, in \emph{Proc. SPIE Wavelets XII
    conf.} (2007)

\bibitem{DY} L. Demanet and L. Ying, Wave atoms and sparsity of
  oscillatory patterns, \emph{Appl. Comput. Harmon. Anal.} {\bf 23-3}
  (2007) 368-387.

\bibitem{WaveatomsTUWE} L. Demanet and L. Ying, Wave atoms and time
  upscaling of wave equations, \emph{to appear in Numer. Math.}
  (2008).

\bibitem{DL1} H. Deng, H. Ling, Fast solution of electromagnetic
  integral equations using adaptive wavelet packet transform,
  \emph{IEEE Trans. Antennas and Propagation}, 47 (4) (1999) 674-682.
  
\bibitem{DL2} H. Deng, H. Ling, On a class of predefined wavelet
  packet bases for efficient representation of electromagnetic
  integral equations, \emph{IEEE Trans. Antennas and Propagation}, 47
  (12) (1999) 1772-1779.

\bibitem{EY1} B. Engquist and L. Ying, Fast directional computation
  for the high frequency Helmholtz kernel in wwo dimensions,
  \emph{submitted}, (2008).
  
\bibitem{EY2} B. Engquist and L. Ying, Fast directional multilevel
  algorithms for oscillatory kernels, \emph{SIAM J. Sci. Comput.},
  {\bf 29-4}, (2007), 1710--1737.
  
\bibitem{G} W.L. Golik, Wavelet packets for fast solution of
  electromagnetic integral equations, \emph{IEEE Trans. Antennas and
  Propagation}, 46 (5) (1998) 618-624.

\bibitem{H} D. Huybrechs and S. Vandewalle, A two-dimensional
  wavelet-packet transform for matrix compression of integral
  equations with highly oscillatory kernel, \emph{J. Comput. Appl.
    Math.}, {\bf 197-1}, (2006), 218--232.
  
\bibitem{Kingsbury} N. Kingsbury, Image processing with complex
  wavelets, \emph{Phil. Trans. Roy. Soc. A} \textbf{357-1760}, (1999),
  2543--2560.

\bibitem{K} S. Kapur and V. Rokhlin, High-order corrected trapezoidal
  quadrature rules for singular functions, \emph{SIAM J. Numer.
    Anal.}, 34 (1997), pp. 1331--1356.
  
\bibitem{Mal} S. Mallat, \emph{A Wavelet Tour of Signal
    Processing}.  Second edition. Academic Press, Orlando-San Diego,
  1999.

\bibitem{Brushlets} F.~G. Meyer and R.~R. Coifman, Brushlets: a tool
  for directional image analysis and image compression, \emph{Applied
    Comput.  Harmon.  Anal.} \textbf{4} (1997), 147--187.

\bibitem{R1} V. Rokhlin, Diagonal forms of translation operators for
  the Helmholtz equation in three dimensions.  \emph{Appl. Comput.
    Harmon. Anal.}  1 (1993), no. 1, 82--93.

\bibitem{R2} V. Rokhlin, Rapid solution of integral equations of
  scattering theory in two dimensions.  \emph{J. Comput. Phys.}  86
  (1990), no. 2, 414--439.

\bibitem{Vil} L. Villemoes, Wavelet packets with uniform
  time-frequency localization, \textit{Comptes-Rendus Mathematique},
  \textbf{335-10} (2002) 793--796.
  
\bibitem{Wat} G. N. Watson, \emph{A Treatise on the Theory of Bessel Functions}, 2nd ed. Cambridge, England: Cambridge University Press, 1966.
\end{thebibliography}
\end{document}